\documentclass[12pt]{amsart}

\usepackage[indent=2em, skip=0.5em]{parskip}

\usepackage{amsmath,amssymb,amsthm}
\usepackage{microtype}
\usepackage[margin = 1in]{geometry}
\usepackage{graphicx}
\usepackage[dvipsnames]{xcolor}
\usepackage{hyperref}
\hypersetup{
    colorlinks=true,
    linkcolor=cyan!80!black,
    citecolor=MidnightBlue,
    urlcolor=magenta,
}
\usepackage[T1]{fontenc}
\usepackage{bm}

\usepackage{subcaption}
\usepackage{booktabs}
\usepackage{multirow}

\theoremstyle{plain}
\newtheorem{theorem}{Theorem}
\newtheorem{proposition}{Proposition}
\newtheorem{lemma}{Lemma}

\newtheorem{definition}{Definition}
\newtheorem{assumption}{Assumption}

\newenvironment{assumption-alt}[1]
  {\renewcommand{\theassumption}{\ref*{#1}$'$}%
   \begin{assumption}}
  {\end{assumption}}

\theoremstyle{remark}
\newtheorem{remark}{Remark}

\newcommand{\Var}{\mathrm{Var}}

\newcommand{\op}{\mathrm{op}}
\newcommand{\comm}{\mathrm{comm}}

\newcommand{\within}{\mathrm{in}}
\newcommand{\across}{\mathrm{out}}
\newcommand{\avg}{\mathrm{avg}}

\DeclareMathOperator{\tr}{Tr}
\DeclareMathOperator{\aic}{AIC}
\DeclareMathOperator{\saic}{soft-AIC}

\DeclareMathOperator{\sbm}{SBM}
\newcommand{\gaic}[1]{\aic^{(#1)}}
\DeclareMathOperator{\scree}{scree}

\newcommand{\cL}{\mathcal{L}}
\newcommand{\cN}{\mathcal{N}}
\newcommand{\bP}{\mathbb{P}}
\newcommand{\bE}{\mathbb{E}}
\newcommand{\bR}{\mathbb{R}}
\newcommand{\bI}{\mathbb{I}}
\newcommand{\convas}{\xrightarrow{\mathrm{a.s.}}}

\newcommand{\semicirc}{\mathrm{sc}}
\newcommand{\tw}{\mathrm{TW}}

\newcommand{\soft}{\mathrm{soft}}

\let\hat\widehat
\let\tilde\widetilde

\title[Model selection in the spiked Wigner model]{Consistent model selection in the spiked Wigner model via AIC-type criteria}
\author[S. S. Mukherjee]{Soumendu Sundar Mukherjee}
\address{Statistics and Mathematics Unit, Indian Statistical Institute, 203 B. T.~Road, Kolkata 700108, India}
\email{ssmukherjee@isical.ac.in}

\begin{document}

\begin{abstract}
Consider the spiked Wigner model
\[
    X = \sum_{i = 1}^k \lambda_i u_i u_i^\top + \sigma G,
\]
where $G$ is an $N \times N$ GOE random matrix, and the eigenvalues $\lambda_i$ are all spiked, i.e. above the Baik-Ben Arous-P\'ech\'e (BBP) threshold $\sigma$. We consider AIC-type model selection criteria of the form
\[
    -2 \, (\text{maximised log-likelihood}) + \gamma \, (\text{number of parameters})
\]
for estimating the number $k$ of spikes. For $\gamma > 2$, the above criterion is strongly consistent provided $\lambda_k > \lambda_{\gamma}$, where $\lambda_{\gamma}$ is a threshold strictly above the BBP threshold, whereas for $\gamma < 2$, it almost surely overestimates $k$. Although AIC (which corresponds to $\gamma = 2$) is not strongly consistent, we show that taking $\gamma = 2 + \delta_N$, where $\delta_N \to 0$ and $\delta_N \gg N^{-2/3}$, results in a weakly consistent estimator of $k$. We further show that a soft minimiser of AIC, where one chooses the least complex model whose AIC score is close to the minimum AIC score, is strongly consistent. Based on a spiked (generalised) Wigner representation, we also develop similar model selection criteria for consistently estimating the number of communities in a balanced stochastic block model under some sparsity restrictions.
\end{abstract}

\maketitle
\thispagestyle{empty}

\section{Introduction}\label{sec:intro}
    Model selection criteria such as the Akaike Information Criterion (AIC) \cite{akaike1998information} or the Bayesian Information Criterion (BIC) \cite{schwarz1978estimating} are staples of classical statistics. It is well-known that in classical fixed dimensional settings, AIC is not consistent in the sense that it tends to select models of higher complexity that the true one. BIC, which adds a more severe complexity penalty, is known to be consistent. See, for instance, \cite{bozdogan1987model, claeskens2008model}.

There is a recent line of work that studies the consistency properties of various classical model selection criteria in high-dimensional settings \cite{fujikoshi2014consistency, Yanagihara2015, yanagihara2015conditions, fujikoshi2016high, bai2018consistency, chakraborty2020high, hu2020detection, bai2022asymptotics}. For instance, \cite{bai2018consistency} showed that in the \emph{spiked covariance model} \cite{johnstone2001distribution} of high-dimensional principal components analysis, where the population covariance matrix has a small number of so-called spiked eigenvalues separated from the rest, AIC is consistent but requires more separation for the spiked eigenvalues than what the so-called Baik-Ben Arous-P\'ech\'e (BBP) threshold demands (loosely speaking, this is a spectral threshold for the spiked eigenvalues below which the behaviour of the extreme eigenvalues of a spiked model resemble those of a non-spiked pure noise model, thereby rendering consistent model selection moot). Later, in \cite{chakraborty2020high} and \cite{hu2020detection}, it was shown that the per parameter penalty in AIC can be modified suitably using results from random matrix theory so that the resulting criterion becomes strongly consistent above any arbitrary spectral threshold above the BBP threshold. Further, it was also shown in \cite{chakraborty2020high} that an estimator may be obtained by tweaking the above-mentioned criterion, which becomes weakly consistent just above the BBP threshold. Note that, in contrast to the fixed dimensional situations, BIC becomes inconsistent. Recently, AIC, BIC, and several other model selection criteria have also been analysed in the context of high-dimensional linear regression \cite{bai2022asymptotics, bai2023koo}.

    In this article, we consider another natural high-dimensional statistical model, the so-called \emph{spiked Wigner model}, where one observes a low-rank signal matrix perturbed (additively) by a Wigner random matrix. There has been a lot of recent interest in the signal detection problem for this model, especially for the rank-one version (an incomplete list of recent works include \cite{perry2018optimality, chung2019weak, jung2020weak, jung2021detection, el2020fundamental, chung2022asymptotic, chung2022weak, jung2023detection, pak2023optimal}). Below the BBP threshold for this model, although consistent detection is not possible, one can still detect the presence of the signal with non-trivial probability \cite{el2020fundamental}. 
In this article, we consider the problem of consistently estimating the number of spiked eigenvalues (i.e. eigenvalues above the BBP threshold), which is essentially a problem of model selection. We ask if classical model selection criteria such as the AIC work in this setting. To that end, consider model selection criteria of AIC type:
\[
    -2 \, (\text{maximised log-likelihood}) + \gamma \, (\text{number of parameters}),
\]
with AIC corresponding to the special case $\gamma = 2$. We refer to the above criterion as $\gaic{\gamma}$.

    Contrary to the spiked covariance model, where AIC is strongly consistent under a certain amount of extra signal above the BBP threshold, we cannot say here that AIC is strongly consistent. However, with a per parameter penalty factor of $2 + \delta_N$, where $\delta_N \to 0$ and $\delta_N \gg N^{-2/3}$, we can show weak-consistency of the resulting criterion $\gaic{2 + \delta_N}$. We also show that a certain soft minimiser of AIC is strongly consistent. In general, for $\gamma > 2$, $\gaic{\gamma}$ is strongly consistent provided $\lambda_k > \lambda_{\gamma}$ where $\lambda_{\gamma}$ is a threshold strictly above the BBP threshold. For $\gamma < 2$, $\gaic{\gamma}$ almost surely overestimates $k$. Our empirical results suggest that these results are valid under more general noise profiles that shun the assumption of independent entries.

    The AIC-type model selection criterion developed here may also be used for estimating the number of communities in networks. We present a stylised application to the stochastic block model with equal sized communities. We utilise the recently developed BBP-type phase transition results for such models in \cite{han2023spectral}. The criterion itself is based on the spiked GOE negative log-likelihood.

We present out theoretical results under a set of assumptions that are satisfied by a number of spiked random matrix models such as the spiked GOE, the spiked Wigner, the spiked generalised Wigner, stochastic block model, etc. Future advances in random matrix theory verifying these assumptions under more general models (such as spiked models with correlated noise profile, general stochastic block models, degree-corrected stochastic block models, etc.) would readily extend our results to such settings.

The rest of the paper is organised as follows. In Section~\ref{sec:set-up}, we discuss the spiked Wigner model and associated random matrix theoretic results, derive the AIC-type criteria explicitly and state our main results. In Section~\ref{sec:SBM}, we develop a stylised application to estimation of the number of communities in networks. In Section~\ref{sec:simu}, we present empirical results comparing the various proposed estimators via several simulation experiments. Finally, Section~\ref{sec:conc} contains a few concluding remarks and future research directions. Section~\ref{sec:proofs} in the Appendix collects all the proofs.

\section{Set-up and main results}\label{sec:set-up}
We first recall the definition AIC for a generic model selection problem. Suppose we have a collection of $q$ candidate (finite-dimensional) statistical models $M_0, \ldots, M_{q - 1}$. Let $d_{M_j}$ denote the dimension of the parameter space under model $M_j$. This is intuitively a simple measure of model complexity. AIC attaches a score
\[
    \aic_j = - 2 \cL_{M_j} + 2 d_{M_j}
\]
to the model $M_j$. Here $\cL_{M_j}$ is the maximised log-likelihood under model $M_j$. In a decision-theoretic set-up, AIC is derived as an unbiased estimator of the risk of the MLE when one measures loss in terms of the Kullback-Liebler divergence. One selects a model $M_{j^*}$ such that
\[
    j^* \in \arg\min_{0 \le j < q} \aic_j.
\]
In this article, we will consider a generalisation of the AIC scores where the per parameter penalty is changed from $2$ to some value $\gamma \ge 0$. We define
\[
    \gaic{\gamma}_j := - 2 \cL_{M_j} + \gamma d_{M_j}.
\]
Thus $\aic_j = \gaic{2}_j$.

\subsection{The spiked Wigner model}
Recall that an $N \times N$ Wigner matrix $W$ is a random symmetric matrix whose diagonal and above-diagonal entries are independent, the above-diagonal entries $W_{ij}, i > j$, having a symmetric law $\mu_1$ with variance $1$ and the diagonal entries $W_{ii}$ having a potentially different symmetric law $\mu_2$ with finite variance. We are interested in the \emph{spiked Wigner model}:
\begin{equation}\label{eq:spiked-Wigner}
    X = A + \frac{\sigma}{\sqrt{N}}W,
\end{equation}
where $A$ is a rank-$k$ positive semi-definite matrix with spectral decomposition \begin{equation}
    A = \sum_{i = 1}^k \lambda_i u_i u_i^\top,
\end{equation}
all whose non-zero eigenvalues are $> \sigma$, the BBP threshold for this model. We assume that $\lambda_1 \ge \cdots \ge \lambda_k$. We want to estimate the unknown rank $k$. Under model $M_j$, we have $k = j$ (under $M_0$, $A \equiv 0$). We denote the eigenvalues of $X$ as $\ell_1 \ge \cdots \ge \ell_N$.

To derive $\gaic{\gamma}$, we need to know the likelihood, which would be different for different distributions of the $W_{ij}$'s. To get around this issue, we will consider normally distributed entries and resort to \emph{universality phenomena} in random matrix theory due to which, under appropriate assumptions, the asymptotic behaviour of the spectrum in certain aspects becomes insensitive to the distribution of the entries. In fact, we will work with matrices from the so-called Gaussian Orthogonal Ensemble (GOE). An $N \times N$ GOE random matrix $G$ is a random symmetric matrix whose upper diagonal entries are i.i.d. $\cN(0, \frac{1}{N})$ and diagonal entries are i.i.d. $\cN(0, \frac{2}{N})$.  Thus we will consider the \emph{spiked GOE model}:
\begin{equation}\label{eq:spiked-GOE}
    X = A + \sigma G,
\end{equation}
under which we will derive $\gaic{\gamma}$ for both the cases $\sigma$ known and $\sigma$ unknown. We will then use the resulting criteria as \emph{a proxy for the actual AIC} for the more general spiked Wigner model \eqref{eq:spiked-Wigner}. This is a common approach in the model selection literature: derive the criterion under a Gaussian noise model and then use that criterion under non-Gaussian noise models as well, the negative log-likelihood under the Gaussian model now serving as a loss function. A non-exhaustive list of works that take this approach in the high-dimensional setting and prove consistency of the resulting criteria under non-Gaussian noise models include \cite{yanagihara2015conditions, bai2018consistency, chakraborty2020high, hu2020detection, bai2022asymptotics}.

The density of $X$ (with respect to the Lebesgue measure on $\bR^{N(N + 1)/2}$) under model \eqref{eq:spiked-GOE} is
\begin{equation}\label{eq:density-X}
    C_N \sigma^{-\frac{N(N + 1)}{2}}e^{-\frac{N}{4 \sigma^2} \tr(X - A)^2},
\end{equation}
where $C_N$ is a normalising constant. The simple form of the density above is the main reason for working with the spiked GOE model instead of some other Gaussian Wigner model.

\begin{assumption}\label{assmp:q-lambda1-fixed}
    We will assume throughout that $q$, the number of candidate models, is bounded and the eigenvalues of the signal matrix $A$, are all fixed (i.e. they do not change with $N$).
\end{assumption}

\begin{remark}
We will prove our consistency results under the assumption that $q$ is fixed. This may be a reasonable assumption if an a priori upper bound on the true model size is available. Anyway, one can get rid of this assumption, provided one has \emph{eigenvalue rigidity estimates} under the model in question. These are technical results from random matrix theory and are not always available in a ready-to-use form. We shall assume that these estimates are available and sketch how our proofs can then be modified to work without any restrictions on $q$.
\end{remark}

\begin{remark}
The assumption that the eigenvalues $\lambda_i$ of $A$ are fixed may be relaxed to the assumption that they converge almost surely to some limits $\eta_i$.
\end{remark}

\subsection{Random matrix theoretic results}
We now recall some relevant random matrix theoretic results on the spiked Wigner model. There is a substantial body of literature surrounding this model. Of particular relevance to us are the works \cite{capitaine2009largest, benaych2011eigenvalues, benaych2011fluctuations}.

It is well-known that bulk empirical spectral measure $\frac{1}{N}\sum_{i = 1}^N \delta_{\ell_i}$ converges weakly almost surely to the semi-circle law whose density is given by
\[
    \varrho_{\semicirc}(x; \sigma^2) = \frac{1}{2\pi \sigma^2} \sqrt{4\sigma^2 - x^2} \, \bI(|x| \le 2\sigma).
\]

To discuss the behaviour of the extreme eigenvalues, we need the function
\[
    \psi_{\sigma}(x) = x + \frac{\sigma^2}{x}.
\]
A plot of this function for $\sigma = 1$ is given in Figure~\ref{fig:graphs}-(a). Note that $\psi_{\sigma}$ achieves its minimum value of $2\sigma$ at $x = \sigma$, to the right of which it is strictly increasing. The first order behaviour of the edge eigenvalues under spiked models is given in the following Assumption (which holds for a wide class of models).

\begin{assumption}\label{assmp:bbp}(BBP-transision)
    We assume that under the true model $M_k$, we have the following:
\begin{enumerate}
    \item[(a)] For $1 \le i \le k$, $\ell_i \convas \psi_{\sigma}(\lambda_i)$.
    \item[(b)] For any fixed $i > k$, we have $\ell_i \convas \psi_{\sigma}(\sigma) = 2\sigma$.
\end{enumerate}
\end{assumption}

\begin{remark}\label{rem:swm-extr}
    Assumption~\ref{assmp:bbp} has been shown to hold
    \begin{enumerate}
        \item [(a)] under the spiked GOE model;
        \item [(b)] under the spiked Wigner model if $\mu_1, \mu_2$ satisfy a Poincar\'{e} inequality (cf. Theorem~2.1 of \cite{capitaine2009largest}), or if under Model $M_k$, the spiked eigenvectors $u_i$'s form a uniformly random $k$-frame (i.e. $k$ mutually orthogonal unit vectors) (cf. Theorem~2.1 of \cite{benaych2011eigenvalues}) (in fact, for the spiked Wigner model, much finer results are known from the work of \cite{Knowles2013} which allow the spikes to be vanishingly close to the BBP threshold of $\sigma$, in the sense that $|\lambda_k - \sigma| \gg N^{-1/3 + \eta}$, $\eta > 0$);
        \item [(c)] under the spiked generalised Wigner model \footnote{A symmetric random matrix $W$ is called a \emph{generalised Wigner matrix} if its upper triangular entries are zero mean independent random variables with a variance profile $\Var(W_{ij}) = \sigma_{ij}^2$, satisfying for each $i$, $\sum_{j} \sigma_{ij}^2 = 1$.} \cite{geng2024outliers};
        \item [(d)] under the stochastic block model with equal community sizes in certain regimes of sparsity \cite{han2023spectral}.
    \end{enumerate}
    One also expects similar results to hold under more general random matrix models, where the entries of the noise matrix are correlated, provided there is sufficient correlation decay. Some preliminary results on such correlated models may be found in \cite{adhikari2019edge, alt2020correlated, banerjee2024edge}.  
\end{remark}

The following (stronger) assumption can be used to remove the restriction in Assumption~\ref{assmp:q-lambda1-fixed} that $q$ is fixed. 
\begin{assumption-alt}{assmp:bbp}\label{assmp:bbp+rigidity}(BBP-transision and eigenvalue rigidity)
    We assume that under the true model $M_k$, there is $q = q_N$ (potentially growing with $N$) such that we have the following:
\begin{enumerate}
    \item[(a)] For $1 \le i \le k$, $\ell_i \convas \psi_{\sigma}(\lambda_i)$.
    \item [(b)] (Rigidity of non-spiked eigenvalues) There exists $\varpi \in (0, 1]$, such that with probability at least $1 - O(N^{-2})$ one has that for all $k < i < q_N$ that
        \[
            |\ell_i - F_{\semicirc}^{-1}(i / N) \sigma| \le \tilde{C}_N \sigma N^{-\varpi},
        \]
        where $F_{\semicirc}$ is the cumulative distribution function of the standard semi-circle law (whose density is $\rho_{\semicirc}(x; 1)$) and $\tilde{C}_N = o(N^{\eta})$ for any $\eta > 0$.
\end{enumerate}
\end{assumption-alt}

Rigidity estimates such as Assumption~\ref{assmp:bbp+rigidity}-(b) have been proved for generalised Wigner matrices (without any spikes) in \cite{erdHos2012rigidity} (with $q_N = N$ and $\varpi = 2/3$). For spiked Wigner models, Theorem 2.7 of \cite{Knowles2013} proves such an estimate (under the assumption of a sub-exponential decay of the tails of the entries of $W$) for $q_N = (\log N)^{\log\log N}$ and $\varpi = 2/3$. Therefore, using Weyl's eigenvalue inequalities, it follows from the two aforementioned results that for our spiked Wigner model, Assumption~\ref{assmp:bbp+rigidity} holds with $q_N = N$ and $\varpi = 2/3$. (In fact, the bulk eigenvalues differ from their classical locations $F_{\semicirc}^{-1}(i / N)$ by at most $O(N^{-1}(\log N)^{\log \log N})$.)

We shall also need the order of fluctuations of the non-spiked extreme eigenvalues for one of our main results.
\begin{assumption}\label{assmp:fluc-order}(Order of fluctuations of the non-spiked eigenvalues)
Assume that under model $M_j$, we have that for any fixed $i > j$,
\[
    N^{2/3}(\ell_i - 2\sigma) = O_P(1).
\]
\end{assumption}

\begin{remark}\label{rem:swm-fluc}
    Assumption~\ref{assmp:fluc-order} has been verified
    \begin{enumerate}
        \item [(a)] under the spiked GOE model (in fact, one also knows the limit distribution: $N^{2/3}(\ell_i - 2\sigma) \xrightarrow{d} \tw_{i - j}$, where $\tw_{\varpi}$ denotes the GOE Tracy-Widom distribution of order $\varpi$);
        \item [(b)] under the spiked Wigner model if $\mu_1$ and $\mu_2$ have sub-exponential tails and under Model $M_j$, the spiked eigenvectors $u_i$'s form a uniformly random $j$-frame (cf. Proposition~5.3 of \cite{benaych2011fluctuations}) (here also, one knows that $N^{2/3}(\ell_i - 2\sigma) \xrightarrow{d} \tw_{i - j}$). In fact, as mentioned earlier, Theorem 2.7 of \cite{Knowles2013} also provides a uniform result of this kind for $i = O((\log N)^{\log\log N})$ under the assumption of a sub-exponential decay of the tails of the entries of $W$.
    \end{enumerate}
\end{remark}

\begin{figure}
    \centering
    \begin{tabular}{cc}
        \includegraphics[scale = 0.25]{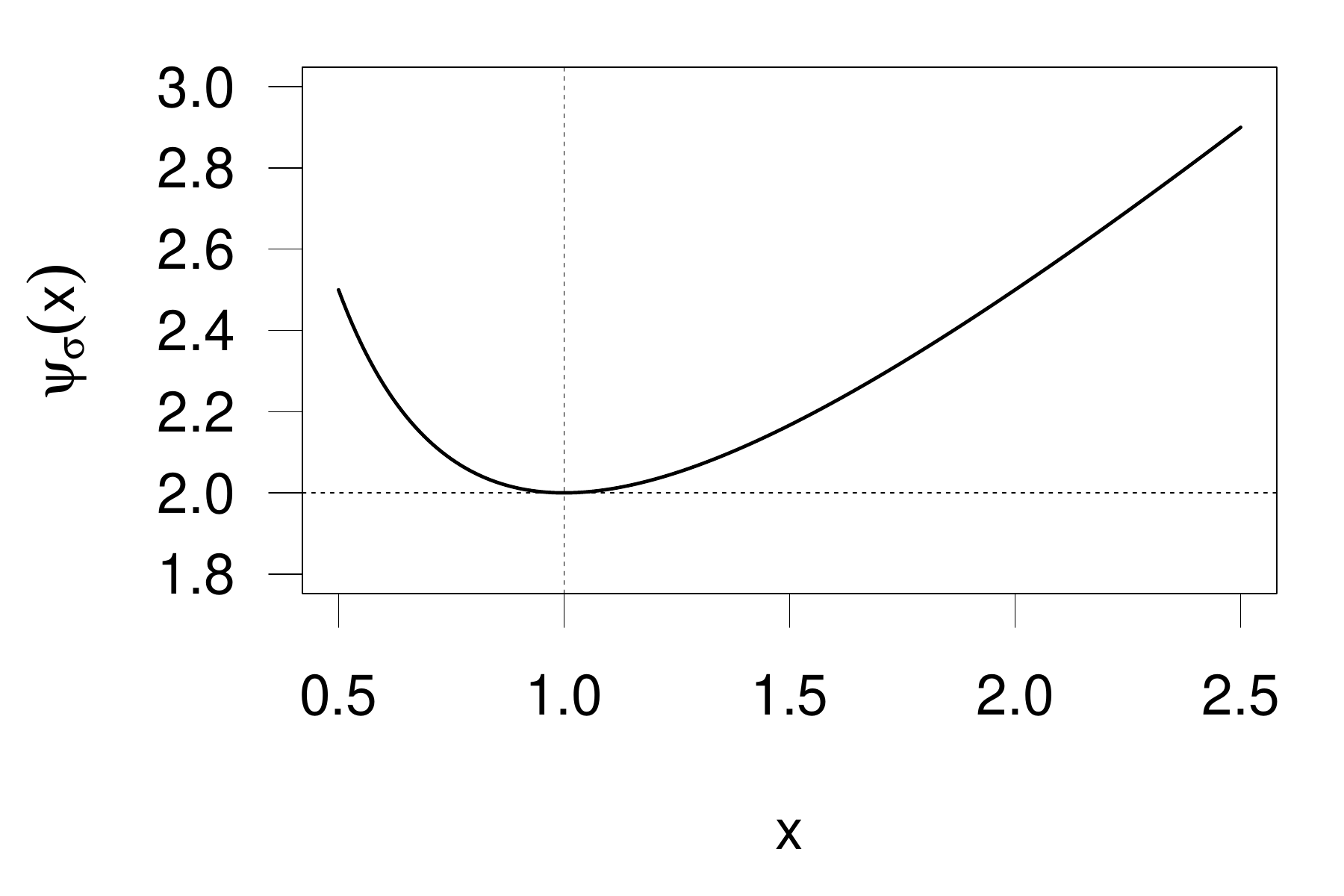} & \includegraphics[scale = 0.25]{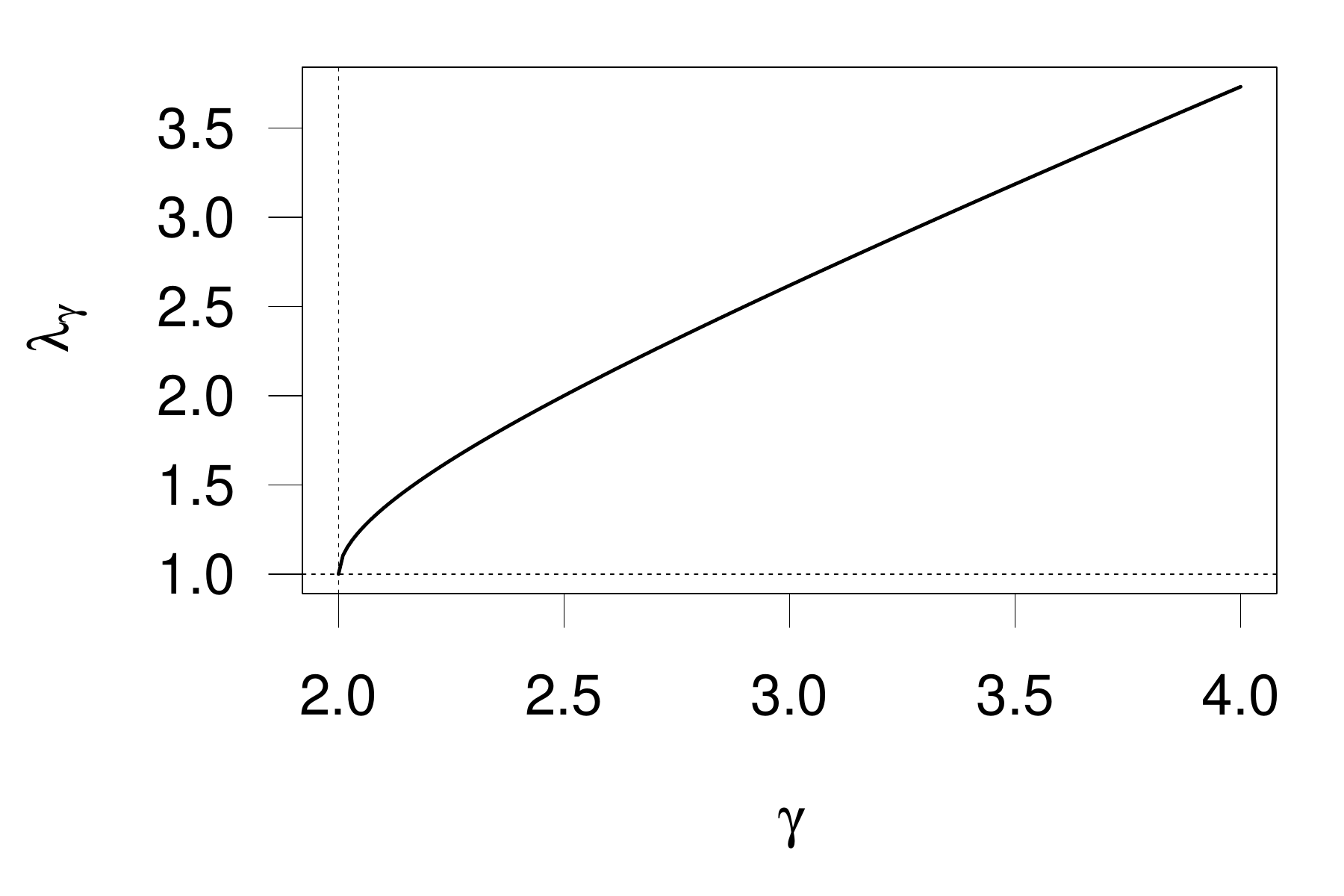} \\
        (a) & (b)
    \end{tabular}
    \caption{(a) Plot of $\psi_{\sigma}(x)$ for $\sigma = 1$. (b) The threshold $\lambda_{\gamma}$ plotted as a function of $\gamma$ for $\sigma = 1$.}
    \label{fig:graphs}
\end{figure}

\subsection{Derivation of \texorpdfstring{$\boldsymbol{\gaic{\gamma}}$}{}}
We now derive exact expressions for the model scores $\gaic{\gamma}_j$ under the spiked GOE model \eqref{eq:spiked-GOE}. Given these scores, we estimate $k$ by minimising $\gaic{\gamma}_j$ over $j \in \{0, 1, \ldots, q - 1\}$:
\begin{align}
    \hat{k}_{\gamma} &:= \arg\min_{0 \le j < q} \gaic{\gamma}_j.
\end{align}

\noindent
\textbf{Known \texorpdfstring{$\sigma$}{}.} Consider first the case of known $\sigma$. The log-likelihood is
\[
    \cL(A) = \log C_N - \frac{N(N + 1)}{4} \log \sigma^2 - \frac{N}{4 \sigma^2} \|X - A\|_F^2.
\]
Let $X = \sum_{i = 1}^N \ell_i v_i v_i^\top$ be the spectral decomposition of $X$. Then the maximum likelihood estimate (MLE) of $A$ is the best rank-$j$ positive semi-definite approximation to $X$ in Frobenius norm, which is given by
\begin{equation}\label{eq:MLE-A}
    \hat{A}_j = \sum_{i = 1}^j \max\{\ell_i, 0\} v_i v_i^\top.
\end{equation}
This is a well-known result (see, for example, Lemma 19 of \cite{clarkson2017low}). Since we will select from a bounded number of candidate models (i.e. $q = O(1)$), for all $1 \le j < q$,
\[
    \liminf_{N \to \infty} \ell_j \ge 2\sigma, \text{ a.s.}
\]
It follows that, almost surely, for $N$ large enough, we have
\begin{equation}\label{eq:MLE-A-simp}
    \hat{A}_j = \sum_{i = 1}^j \ell_i v_i v_i^\top.
\end{equation}
In the case $j = 0$, we take $\hat{A}_0 = 0$, the zero matrix. Hence
\begin{align*}
    \cL(\hat{A}_j) &= \log C_N - \frac{N(N + 1)}{4} \log \sigma^2 - \frac{N}{4\sigma^2} \|X - \hat{A}_j\|_F^2 \\
                    &= \log C_N - \frac{N(N + 1)}{4} \log \sigma^2 - \frac{N}{4\sigma^2} \sum_{i > j} \ell_i^2.
\end{align*}
Therefore, in the case of known $\sigma$, we have
\begin{align} \nonumber
    \gaic{\gamma}_j &= -2 \cL(\hat{A}_j) + \gamma \bigg(N j - \frac{j(j - 1)}{2}\bigg) \\
                    &= - 2\log C_N + \frac{N(N + 1)}{2} \log \sigma^2 + \frac{N}{2 \sigma^2} \sum_{i > j} \ell_i^2 + \gamma \bigg(N j - \frac{j(j - 1)}{2}\bigg). \label{eq:gaic-known-sigma}
\end{align}

\noindent
\textbf{Unknown \texorpdfstring{$\sigma$}{}.} The log-likelihood is
\[
    \cL(A, \sigma^2) = \log C_N - \frac{N(N + 1)}{4} \log \sigma^2 - \frac{N}{4 \sigma^2} \|X - A\|_F^2.
\]
As in the case of known $\sigma$, the maximum likelihood estimate (MLE) of $A$ under model $M_j$ is given by \eqref{eq:MLE-A-simp} (almost surely, for large enough $N$). The MLE of $\sigma^2$ under model $M_j$ is given by
\[
    \hat{\sigma^2_j} = \frac{1}{N + 1} \|X - \hat{A}_j\|_F^2 = \frac{1}{N + 1}\sum_{i > j} \ell_i^2.
\]
Thus
\begin{align*}
    \cL(\hat{A}_j, \hat{\sigma^2_j}) &= \log C_N - \frac{N(N + 1)}{4} \log \hat{\sigma^2_j} - \frac{N}{4\hat{\sigma^2_j}} \|X - \hat{A}_j\|_F^2 \\
                                    &= \log C_N - \frac{N(N + 1)}{4} \log \hat{\sigma^2_j} - \frac{N}{4\hat{\sigma^2_j}} \sum_{i > j} \ell_i^2,
\end{align*}
and
\begin{align}\nonumber
    \gaic{\gamma}_j &= -2 \cL(\hat{A}, \hat{\sigma^2_j}) + \gamma \bigg(1 + N j - \frac{j(j - 1)}{2}\bigg) \\
                    &= - 2\log C_N + \frac{N(N + 1)}{2} \log \hat{\sigma^2_j} + \frac{N}{2 \hat{\sigma^2_j}} \sum_{i > j} \ell_i^2 + \gamma \bigg(1 + N j - \frac{j(j - 1)}{2}\bigg). \label{eq:gaic-uknown-sigma}
\end{align}

\subsection{Results on \texorpdfstring{$\boldsymbol{\hat{k}_{\gamma}}$}{}}
We are now ready to state our main results on the selection properties of $\hat{k}_{\gamma}$. We will distinguish between two notions of consistency.
\begin{definition}[Consistency]\label{def:consistency}
    An estimator $\hat{k}$ of $k$ is called \emph{strongly consistent} if $\hat{k} \convas k$. It is called \emph{weakly consistent} if $\bP(\hat{k} = k) \to 1$.
\end{definition}
\begin{theorem}\label{thm:gamma-aic}
    Suppose that Assumptions~\ref{assmp:q-lambda1-fixed} and \ref{assmp:bbp} hold. Regardless of whether $\sigma$ is known or unknown, we have the following:
    \begin{enumerate}
    \item [(a)] If $\gamma \le 2$, then almost surely, $\liminf_{N \to \infty} \hat{k}_{\gamma} \ge k$.
    
    \item [(b)] If $\gamma > 2$, then almost surely, $\limsup_{N \to \infty} \hat{k}_{\gamma} \le k$.

    \item [(c)] Further, if $\lambda_k > \lambda_{\gamma} := \psi_{\sigma}^{-1}(\sqrt{2\gamma}\sigma)$, then for $\gamma > 2$, almost surely, $\liminf_{N \to \infty} \hat{k}_{\gamma} \ge k$.
    \end{enumerate}
    As a consequence, if $\lambda_k > \lambda_{\gamma}$, then $\hat{k}_{\gamma}$ is strongly consistent for $k$.
\end{theorem}

\begin{remark}
    Under Assumption~\ref{assmp:bbp+rigidity} with $q_N = N$, we have the conclusions of Theorem~1 without the assumption of the boundedness of $q$, the number of candidate models.
\end{remark}

In Figure~\ref{fig:graphs}-(b), we plot the threshold $\lambda_{\gamma}$ as a function of $\gamma$ for $\sigma = 1$. This is strictly bigger than the BBP threshold $\psi_{\sigma}^{-1}(2 \sigma) = \sigma$. Thus, as far as consistent selection is concerned, choosing a penalty factor $\gamma > 2$ is suboptimal.

Note also that Theorem~\ref{thm:gamma-aic} does not say anything about the consistency of AIC (i.e. the case $\gamma = 2$). In fact, a look at its proof (given in Section~\ref{sec:proofs}) reveals that the scores $\aic_j$, $k \le j < q$ become asymptotically of the same order. As the true model $M_k$ is the least complex among the models $M_j, k \le j < q$, there is hope that one may be able to identify the true model by selecting the least complex model close to the minimiser of $\aic_j$. Under additional assumptions, this is indeed possible (see Section~\ref{sec:soft-aic} below).

Since $\aic$ is at the borderline of the dichotomy revealed in parts (a) and (b) of Theorem~\ref{thm:gamma-aic}, a natural question is if we can use a penalty factor $\gamma_N = 2 + \delta_N$, where $\delta_N$ shrinks to $0$ at an appropriate rate and achieve consistency. Using the $N^{-2/3}$ fluctuations of the non-spiked extreme eigenvalues (cf. Assumption~\ref{assmp:fluc-order}), we can establish the following weak consistency result.
 
\begin{theorem}\label{thm:almost-aic-weak-consistency}
    Suppose that Assumptions~\ref{assmp:q-lambda1-fixed},~\ref{assmp:bbp} and \ref{assmp:fluc-order} hold. Let $\delta_N$ be a sequence such that $\delta_N \to 0$ and $\delta_N \gg N^{-2/3}$. Then $\hat{k}_{2 + \delta_N}$ is weakly consistent for $k$, i.e. $\bP(\hat{k}_{2 + \delta_N} = k) \to 1$. This is true regardless of whether $\sigma$ is known or unknown.
\end{theorem}

\begin{remark}
    A natural question left unanswered here is if there is a choice of $\delta_N$ for which $\hat{k}_{2 + \delta_N}$ is strongly consistent.
\end{remark}

\subsection{The scree plot estimator}
In spiked models such as PCA, one typically first looks at a plot of the eigenvalues for a kink therein. Such plots are commonly referred to as scree plots. In the context of our spiked model, one can construct such an estimator as follows (we will refer to this as the scree plot estimator):
\[
    \hat{k}_{\scree} = \sup\bigg\{0 \le j < q : \frac{\ell_j}{2\hat{\sigma}} > 1\bigg\},
\]
where $\hat{\sigma}$ is a strongly consistent estimate of $\sigma$. This will serve as a benchmark estimator.

\begin{proposition}\label{prop:scree}
   Suppose that Assumptions~\ref{assmp:q-lambda1-fixed} and \ref{assmp:bbp} hold. Then $\hat{k}_{\scree}$ is strongly consistent for $k$.
\end{proposition}

\subsection{A soft-minimisation approach}\label{sec:soft-aic}
Although from Theorem~\ref{thm:gamma-aic}, we cannot guarantee strong consistency of AIC, as discussed earlier, there is some hope of recovering the true model by selecting the least complex model close to the minimiser of AIC. We now make this precise.

For a threshold $\hat{\xi} > 0$, consider the following estimator of $k$ which we dub $\saic$:
\begin{equation}\label{eq:soft-aic}
    \hat{k}_{2, \,\soft} := \min\bigg\{0 \le j < q : |\aic_j - \min_{0 \le j' < q} \aic_{j'}| < \frac{\hat{\xi}}{3}\bigg\}.
\end{equation}
The threshold $\hat{\xi}$ has to be chosen carefully so that in the minimisation above models $M_j$ with $j < k$ are automatically discarded. To that end, define for $j \le k$,
\[
    \xi_j := \frac{1}{2\sigma^2}(\psi_{\sigma}(\lambda_j)^2 - 4\sigma^2).
\]
In effect, our threshold $\hat{\xi}$ should be smaller than $\xi_k$ in an appropriate sense.

\begin{theorem}\label{thm:soft-aic}
    Suppose that Assumptions~\ref{assmp:q-lambda1-fixed} and \ref{assmp:bbp} hold and we can construct a threshold $\hat{\xi}$ such that almost surely,
    \begin{equation}\label{eq:soft-aic-thres-prop}
        0 < \liminf_{N \to \infty} \hat{\xi} \le \limsup_{N \to \infty} \hat{\xi} \le \xi_k.
    \end{equation}
    Then $\hat{k}_{2, \,\soft} \convas k$.
\end{theorem}

\noindent
\textbf{Construction of a suitable threshold \texorpdfstring{$\hat{\xi}$}{}.}
How do we construct a threshold $\hat{\xi}$ such that \eqref{eq:soft-aic-thres-prop} holds? For a strongly consistent estimator $\hat{\sigma^2}$ of $\sigma^2$, set
\[
    \hat{\xi}_j := \frac{1}{2\hat{\sigma^2}}(\ell_j^2 - 4\hat{\sigma^2}) \convas \begin{cases}
        \xi_j & \text{ if } j \le k, \\
        0 & \text{ if } k < j < q.
\end{cases}
\]
Since $\lambda_1$ is bounded, then so is $\frac{\xi_1}{\xi_k}$. Assume that we \emph{know} an a priori upper bound $B$ on $\frac{\xi_1}{\xi_k}$, i.e. $\xi_1 \le B \xi_k$. For instance, in the equal-spikes case (i.e. $\lambda_1 = \cdots = \lambda_k$), we can take $B = 1$. Then, since $q$ is bounded, we may take
\begin{equation}\label{eq:xi-hat}
    \hat{\xi} = \frac{1}{qB} \sum_{j = 1}^q \hat{\xi}_j \convas \xi:= \frac{1}{qB} \sum_{j = 1}^k \xi_j.
\end{equation}
Clearly,
\[
     0 < \xi \le \frac{kB\xi_k}{qB} = \frac{k}{q} \xi_k \le \xi_k,
\]
so that \eqref{eq:soft-aic-thres-prop} holds for $\hat{\xi}$.

\noindent
\textbf{Data-driven soft minimisation in two steps.}\label{sec:two-step-soft-min-AIC}
Since it may be unrealistic to assume knowledge of a value of the upper bound $B$, we may adopt a two-step procedure. First we estimate $B$ in a data-driven way as follows: we determine the value of $k$ using the scree plot estimator $\hat{k}_{\scree}$. After that, we calculate the estimate $\hat{ \xi}_{\hat{k}_{\scree}}$ and $\hat{\xi}_1$ of $\xi_k$ and $\xi_1$, respectively. Next, we take the ratio of these two quantities and use the resulting value as an estimate for $B$. Once we have the estimate for $B$, we can employ the soft minimization procedure.

\begin{theorem}\label{thm:two-step}
    Suppose that Assumptions~\ref{assmp:q-lambda1-fixed} and \ref{assmp:bbp} hold. Let $\hat{k}_{\scree}$ denote the scree plot estimator. Set
    \[
        \hat{B} = \frac{\hat{\xi}_1}{\hat{\xi}_{\hat{k}_{\scree}}} \quad \text{and} \quad
        \hat{\xi}_{\scree} = \frac{1}{q\hat{B}} \sum_{j = 1}^q \hat{\xi}_j.
    \]
    Then $\hat{k}_{2, \soft}$ constructed with the threshold $\hat{\xi}_{\scree}$ is strongly consistent for $k$.
\end{theorem}
\begin{remark}
    Instead of the scree plot estimator, one may use any other strongly consistent estimator of $k$. Note also that in the construction of $\hat{\xi}_{\scree}$, one may just average over the $\hat{k}_{\scree}$ many top eigenvalues instead of the top $q$. This is useful because this construction would then work when $q$ is allowed to grow with $N$.
\end{remark}

\subsection{Estimating \texorpdfstring{$\boldsymbol{\sigma}$}{}.}
Instead of using the $\gaic{\gamma}$ scores \eqref{eq:gaic-uknown-sigma} for unknown $\sigma$, one may also use the scores \eqref{eq:gaic-known-sigma} for known $\sigma$ by plugging in some strongly consistent estimate of $\sigma$. We also need such an estimator for the proposed construction \eqref{eq:xi-hat} of the threshold $\hat{\xi}$ in $\saic$ and for the scree plot estimator $\hat{k}_{\scree}$. Technically, we could use
\[
    \hat{\sigma^2_0} = \frac{1}{N + 1}\sum_{j = 1}^N \ell_j^2 = \frac{1}{N + 1} \|X\|_F^2,
\]
which converges almost surely to $\sigma^2$. However, if some of the top spiked eigenvalues are large, then for small to moderate $N$, $\hat{\sigma^2_0}$ tends to overestimate $\sigma^2$, leading to poor performance (see Section~\ref{sec:simu} for empirical demonstrations of this phenomenon). Instead, we can throw away some of the extreme eigenvalues and adjust the estimate accordingly. For $\alpha \in (0, 1)$, let $q_{\alpha}$ denote the top $\alpha$-th quantile of the standard semi-circle law, i.e.
\[
    \alpha = \int_{q_{\alpha}}^2 \varrho_{\semicirc}(x; 1) \, dx.
\]
Then we may take the following trimmed estimator of $\sigma^{2}$:
\[
    \hat{\sigma}^2_{\alpha} := \frac{\frac{1}{N}\sum_{j : \ell_{\lfloor \alpha N \rfloor} \le \ell_j \le \ell_{\lfloor (1 - \alpha) N \rfloor}} \ell_j^2} {\int_{-q_{\alpha}}^{q_{\alpha}} x^2 \varrho_{\semicirc}(x; 1) \, dx}.
\]
This also converges almost surely to $\sigma^2$.

\section{Estimating the number of communities under a Stochastic Block Model}\label{sec:SBM}
While we presented and analysed our criteria for the spiked Wigner model, it can also be used to estimate the number of communities under a Stochastic Block Model (SBM). The adjacency matrix of a random graph sampled from an SBM may be thought of as a low-rank perturbation to a Wigner matrix with a variance profile. Rigorous BBP-type phase transition results for such matrices are not available in full generality. A recent preprint \cite{han2023spectral} works out a BBP-type result for stochastic block models with equal-sized communities. In this section, we will write down a criterion based on \eqref{eq:gaic-known-sigma} and analyse its behaviour under the above-mentioned model using the results of \cite{han2023spectral}.

In our SBM, $n$ nodes are partitioned into $k$ equal-sized communities. Nodes from the same community form connections with probability $p_{\within}$ and nodes from different communities do so with probability $p_{\across}$, with the connections being formed independent of one another. We denote this model as $\sbm(N, k; p_{\within}, p_{\across})$. Let $A$ denote the adjacency matrix of a graph sampled from this model. Note that under this model, $p_{\avg} = \frac{p_{\within} + (k - 1) p_{\across}}{k}$ is the average connection probability.

It may then be checked that $\bE[A]$ is of rank $k$ and its non-zero eigenvalues are $\lambda_1 = N p_{\avg}$ with multiplicity $1$ and $\lambda_2 = \frac{N(p_{\within} - p_{\across})}{k}$ with multiplicity $k - 1$. Moreover, one may write
\[
    \bE[A] = p_{\avg} J_N + \frac{N(p_{\within} - p_{\across})}{k} \sum_{i = 2}^k w_i w_i^\top,
\]
where $J_N$ denotes the $N \times N$ matrix of all $1$'s, and $w_2, \ldots, w_k$ are $(k - 1)$ orthonormal eigenvectors corresponding to $\lambda_2$.

Let
\[
    \varsigma := \sqrt{N \cdot \frac{p_{\within} (1 - p_{\within}) + (k - 1) p_{\across} (1 - p_{\across})}{k}}; \qquad \theta_N := \frac{N (p_{\within} - p_{\across})}{\varsigma k}.
\]
Consider the matrix
\[
    M := \frac{(A - p_{\avg} J_N)}{\varsigma}.
\]
Notice that
\[
    \bE[M] = \frac{\bE[A] - p_{\avg} J_N}{\varsigma} = \frac{\lambda_2}{\varsigma} \sum_{j = 2}^{k - 1} w_j w_j^\top = \theta_N \sum_{j = 2}^{k - 1} w_j w_j^\top.
\]
Note also that
\[
    \sum_{j} \Var(M_{ij}) = \frac{1}{\varsigma^2 N} \sum_{j} \Var(A_{ij}) = \frac{1}{\varsigma^2 N} \bigg(\frac{N}{k} p_{\within}(1 - p_{\within}) + \frac{N (k - 1)}{k} p_{\across}(1 - p_{\across})\bigg) = \frac{1}{N}.
\]
Recall that a \emph{generalised Wigner matrix} is a random symmetric matrix $W$ whose entries are zero mean independent random variables such that $\sum_{j} \Var(W_{ij}) = 1$ for each $i$. Therefore, we may regard the matrix $M$ as a rank-$(k - 1)$ perturbation of a generalised Wigner matrix $W = \sqrt{N}(M - \bE[M])$:
\[
    M = \theta_N \sum_{j = 1}^N w_i w_i^\top + \frac{1}{\sqrt{N}} W.
\]
This is exactly of the form~\eqref{eq:spiked-Wigner}, with $\sigma^2 = 1$. We now recall the following result from \cite{han2023spectral}, restated in our notation.
\begin{proposition}[Theorem 2.1 and Corollary~3.2 of \cite{han2023spectral}]\label{prop:BBP_M}
    Suppose that $p_{\avg} \gg N^{-3/4}$ and $\theta_N \to \theta \in (0, \infty]$. Then
    \begin{enumerate}
        \item[(a)] If $\theta < 1$, then $\lambda_i(M) \convas 2$ for any fixed $i \ge 1$.
        \item[(b)] If $\theta > 1$, then $\lambda_i(M) \convas 2$ for any fixed $i \ge k$ and
            \[
                \lambda_i - (\theta_N + \frac{1}{\theta_N}) \convas 0 \quad 1 \le i \le k - 1.
            \]
    \end{enumerate}
\end{proposition}
This gives a BBP-type phase transition result for $M$ and thus we are tempted to use the model selection criteria developed earlier on $M$. However, we do not know $\varsigma$ or $p_{\avg}$. We can estimate the latter using
\[
    \hat{p}_{\avg} := \frac{1}{N^2} \sum_{i, j} A_{ij}.
\]
Then to estimate $\varsigma$, we will use
\[
    \hat{\varsigma} = \sqrt{N \hat{p}_{\avg} (1 - \hat{p}_{\avg})}.
\]
With these estimators in hand, we shall consider the matrix
\[
    \hat{M} := \frac{(A - \hat{p}_{\avg} J)}{\hat{\varsigma}}.
\]
The following proposition compares the matrices $M$ and $\hat{M}$ in the operator norm.
\begin{proposition}\label{prop:M-vs-Mhat}
    Suppose that $N^{-2\beta} \ll p_{\avg} \ll 1 - \delta_0$ with $\beta \in (0, 1/2)$ and some fixed $\delta_0 \in (0, 1)$. Assume that (i) $\theta_N \ll N^{1/2}$ and (ii) $\limsup_{N \to \infty} N^{-1/2}\|M - \bE M\|_{\op} < \infty$ a.s. Then
\[
    |\frac{\hat{\varsigma}}{\varsigma} - 1| \convas 0 \qquad \text{and} \qquad \theta_N^{-1} \|\hat{M} - M\|_{\op} \convas 0.
\]
\end{proposition}
A few remarks are in order. First, the assumption $\theta_N \ll N^{1/2}$ is needed to show the closeness of $\hat{\varsigma}$ and $\varsigma$ (and a fortiori that of $\hat{M}$ and $M$). Second, the assumption that $\limsup_{N \to \infty} N^{-1/2} \|M - \bE M\|_{\op} < \infty$ is a rather mild one. In fact, since $M - \bE M$ is a generalised Wigner matrix, we expect that under mild conditions,
\begin{equation}\label{eq:op-norm-Mhat-minus-M}
    \|M - \bE M\|_{\op} \convas 2.
\end{equation}
Anyway, we will be operating under the assumption $p_{\avg} \gg N^{-3/4}$ of Proposition~\ref{prop:BBP_M}, under which one may readily show that $\|M - \bE M\|_{\op} = O(1)$ with (polynomially) high probability (see, e.g., Theorem~5.2 of \cite{lei2015consistency}).

From Weyl's inequality and Proposition~\ref{prop:BBP_M} it follows that
\begin{proposition}\label{prop:BBP_Mhat}
    Suppose $p_{\avg} \gg N^{-3/4}$. Suppose that $\theta_N \to \theta \in (0, \infty]$ and $\theta_N \ll N^{1/2}$. Then
    \begin{enumerate}
        \item[(a)] If $\theta \le 1$, then $\ell_i(\hat{M}) \convas 2$ for any fixed $i \ge 1$.
        \item[(b)] If $1 < \theta < \infty $, then $\ell_i(\hat{M}) \convas 2$ for any fixed $i \ge k$ and
            \[
                \ell_i(\hat{M}) \convas (\theta + \frac{1}{\theta}) \quad \text{for } 1 \le i \le k - 1.
            \]
        \item[(c)] If $1 \ll \theta_N \ll N^{1/2}$, then almost surely, for any $1 \le i \le k - 1$,
            \[
                \ell_i(\hat{M}) - (\theta_N + \frac{1}{\theta_N}) = o(\theta_N).
            \]
            Further, for any fixed $i \ge k$, $\ell_i(\hat{M}) \convas 2$.
    \end{enumerate}
\end{proposition}
This motivates us to define the following AIC-like criterion for selecting the number of blocks in SBM:
\[
    H_{\gamma}(j) = \frac{1}{2} \sum_{i > j} \ell_i(\hat{M})^2 - \gamma(N - j)\bigg(1 - \frac{N + j - 1}{2N}\bigg).
\]
(This is nothing but $\frac{1}{N}(\gaic{\gamma}_j - \overline{C}_N)$ for a quantity $\overline{C}_N$ depending only on $N$, with $\gaic{\gamma}_j$ as in \eqref{eq:gaic-known-sigma} and $\sigma^2 = 1$). We propose to estimate $k$ by
\begin{equation}\label{eq:aic_comm}
    \hat{k}_{\gamma}^{(\comm)} = 1 + \arg\min_{0 \le j < q} H_{\gamma}(j).
\end{equation}
(We may also define the scree plot estimator and $\saic$ or its two-step data-driven version analogously using the sample eigenvalues of $\hat{M}$.) The following is the analogue of Theorem~\ref{thm:gamma-aic} for $\hat{k}_{\gamma}^{(\comm)}$.
\begin{theorem}\label{thm:gamma-aic-comm}
    Suppose that $A$ is the adjacency matrix of a graph sampled from $\mathrm{SBM}(N, k; p_{\within}, p_{\across})$, where $p_{\avg} \gg N^{-3/4}$,  $\theta_N \to \theta \in (0, \infty]$ and $\theta_N \ll N^{1/2}$. Then the following results hold for the model selection criterion \eqref{eq:aic_comm}.
    \begin{enumerate}
        \item [(a)] If $\gamma \le 2$, then almost surely, $\liminf_{N \to \infty} \hat{k}_{\gamma}^{(\comm)} \ge k$.
    
        \item [(b)] If $\gamma > 2$, then almost surely, $\limsup_{N \to \infty} \hat{k}_{\gamma}^{(\comm)} \le k$.

    \item [(c)] Further, if $\theta > \psi_{1}^{-1}(\sqrt{2\gamma})$, then for $\gamma > 2$, almost surely, $\liminf_{N \to \infty} \hat{k}_{\gamma} \ge k$.
    \end{enumerate}
    As a consequence, if $\theta > \psi_{1}^{-1}(\sqrt{2\gamma})$, then $\hat{k}_{\gamma}^{(\comm)}$ is strongly consistent for $k$.
\end{theorem}
We also have the obvious analogues of Theorems~\ref{thm:soft-aic} and~\ref{thm:two-step} in this setting. For brevity, we omit these.

\section{Empirical results}\label{sec:simu}
In this section\footnote{The results of this section can be replicated using the R code available at \protect{\url{https://gitlab.com/soumendu041/swm-aic}}.}, we report simulation experiments comparing $\aic$, $\gaic{2 + \delta_N}$, $\gaic{\gamma}$, $\saic$, the scree plot estimator and $\saic^{\mathrm{(ad)}}$, the data-driven version of $\saic$ with the scree plot estimator providing a pilot estimate of $k$. We report two performance metrics for each estimator $\hat{k}$ based on Monte Carlo replications: (i) average dimensionality of the selected model along with the corresponding standard error estimates and an estimate of the \emph{probability of correct selection (PCS)}, i.e. $\bP(\hat{k} = k)$.

For the AIC-type estimators, we use the scores \eqref{eq:gaic-known-sigma} for known $\sigma$ together with various estimates of $\sigma$, and also the scores \eqref{eq:gaic-uknown-sigma} for unknown $\sigma$. For the scree plot estimator, we use estimates of $\sigma$. Below S-1 refers to the situation, where we use the \emph{oracle} value of $\sigma$ in the estimators. S-2 (resp. S-3) refers to the situation where we use $\hat{\sigma^2_0}$ (resp. $\hat{\sigma^2_{\alpha}}$) as an estimate of $\sigma$. Finally, S-4 refers to the situation (applicable to the AIC-type estimators only) where the scores corresponding to unknown $\sigma$ are used.

In the experiments below, we take $\sigma^2 = 1$, $N = 1000$, $\gamma = 2.15$. The spiked matrix $A$ is taken to be a diagonal matrix with specified eigenvalues $\lambda_1, \ldots, \lambda_k, 0, \ldots, 0$. We use $B = \frac{\xi_1}{\xi_k} + 1$ for $\saic$, $\delta_N = \frac{0.1}{\sqrt{N}}$ and take $\alpha = 0.1$ in $\hat{\sigma^2_{\alpha}}$. For $\saic^{(\mathrm{ad})}$, we use $\hat{B} = \frac{\hat{\xi}_1}{\hat{\xi}_{\hat{k}_{\scree}}} + 1$, with $\hat{\sigma^2_{\alpha}}$ used for computing the $\hat{\xi}_j$'s. We search over $q = 20$ candidate models. The performance metrics are all based on $100$ Monte Carlo replications.

We consider three different noise profiles: (i) GOE; (ii) Wigner with Rademacher entries (i.e. random signs); and (iii) Schur-Hadamard product of independent symmetric Toeplitz and Hankel random matrices each with i.i.d. $\cN(0, 1)$ entries, with the resulting matrix being scaled by $\sqrt{N}$. In Tables~\ref{table:1} and \ref{table:2}, these profiles are denoted by ``GOE'', ``Rad.'' and ``$T\odot H$'', respectively. Note that profile (iii) has dependent (but uncorrelated) entries. The empirical eigenvalue distribution of this profile was studied in \cite{bose2014bulk, mukherjee2022convergence} where it was shown to converge to the semi-circle law. We expect that the results on the behaviour of edge eigenvalues (stated under the spiked Wigner model) would also hold for this profile (as an instantiation of universality). Therefore it is natural to anticipate that the selection properties of the various estimators considered would be valid under this noise profile as well. This is indeed confirmed in our experiments.

In our first experiment reported in Table~\ref{table:1}, we have $k = 4$ spiked eigenvalues, with $\lambda_4 = 1.1$, just above the BBP threshold $1$. Since $\lambda_{\gamma} = 1.31$, $\gaic{\gamma}$ always underestimates $k$. $\aic$, $\gaic{2 + \delta_N}$, and $\saic$ all perform much better than the simple scree plot estimator. Further, all of these three AIC-type estimators perform the best in S-4, where the scores for unknown $\sigma^2$ are used. Also, note that using $\hat{\sigma^2_0}$ as an estimator of $\sigma$ results in disastrous performance, likely due to the effect of the relatively large top eigenvalue $\lambda_1 = 5$. The trimmed estimator $\hat{\sigma^2_{\alpha}}$ fares much better in comparison. Somewhat surprisingly, under the $T\odot H$ noise profile, all the estimators suffer from degraded performance in S-1.

\begin{table}[!t]
    \centering
    \setlength{\tabcolsep}{0.22em}
    \caption{$k = 4$, $(\lambda_1, \lambda_2, \lambda_3, \lambda_4)^\top = (5, 1.5, 1.2, 1.1)^\top$, $\sigma^2 = 1$. $N = 1000$, $\delta_N = \frac{0.1}{\sqrt{N}}$, $\lambda_{2 + \delta_N} = 1.04$, $\gamma = 2.15$, $\lambda_{\gamma} = 1.31$. S-1: oracle $\sigma^2$; S-2: $\hat{\sigma^2_0}$; S-3: $\hat{\sigma^2_{\alpha}}$ with $\alpha = 0.1$; S-4: unknown $\sigma^2$.}
    \label{table:1}
    \fontsize{6}{8}\selectfont
    \begin{tabular}{l|l*{4}{c}|*{4}{c}|*{4}{c}|*{4}{c}|*{3}{c}|*{4}{c}}
        \toprule
        \multicolumn{2}{c}{} & \multicolumn{4}{c}{$\boldsymbol{\aic}$} & \multicolumn{4}{c}{$\boldsymbol{\gaic{2 + \delta_N}}$} & \multicolumn{4}{c}{$\boldsymbol{\saic}$} & \multicolumn{4}{c}{$\boldsymbol{\gaic{\gamma}}$} & \multicolumn{3}{c}{$\boldsymbol{\scree}$} & \multicolumn{4}{c}{$\boldsymbol{\saic^{(\mathrm{ad})}}$} \\
        \midrule
        \multicolumn{2}{c}{} & S-1 & S-2 & S-3 & S-4 & S-1 & S-2 & S-3 & S-4 & S-1 & S-2 & S-3 & S-4 & S-1 & S-2 & S-3 & S-4 & S-1 & S-2 & S-3 & S-1 & S-2 & S-3 & S-4 \\
        \midrule
        \multirow{3}{*}{\rotatebox{90}{\bfseries{GOE}}}
        & \bfseries{mean} & 3.44 & 2.68 & 3.23 & 3.78 & 3.40 & 2.67 & 3.21 & 3.67 & 3.44 & 2.68 & 3.22 & 3.78 & 2.11 & 1.99 & 2.04 & 2.13 & 3.36 & 2.66 & 3.18 & 3.25 & 2.61 & 3.09 & 3.50 \\
        & \bfseries{sd}   & 0.54 & 0.57 & 0.51 & 0.56 & 0.53 & 0.57 & 0.54 & 0.60 & 0.54 & 0.57 & 0.52 & 0.56 & 0.31 & 0.17 & 0.20 & 0.34 & 0.52 & 0.55 & 0.54 & 0.56 & 0.53 & 0.49 & 0.58 \\
        & \bfseries{PCS}  & 0.46 & 0.05 & 0.27 & 0.64 & 0.42 & 0.05 & 0.27 & 0.56 & 0.46 & 0.05 & 0.27 & 0.64 & 0.00 & 0.00 & 0.00 & 0.00 & 0.38 & 0.04 & 0.25 & 0.31 & 0.02 & 0.17 & 0.48 \\
        \midrule
        \multirow{3}{*}{\rotatebox{90}{\bfseries{Rad.}}}
        & \bfseries{mean} & 3.36 & 2.62 & 3.12 & 3.68 & 3.29 & 2.57 & 3.07 & 3.65 & 3.34 & 2.61 & 3.12 & 3.68 & 2.01 & 2 & 2.01 & 2.04 & 3.24 & 2.54 & 3.02 & 3.16 & 2.46 & 2.93 & 3.51 \\
        & \bfseries{sd}   & 0.50 & 0.49 & 0.46 & 0.63 & 0.52 & 0.50 & 0.43 & 0.61 & 0.52 & 0.49 & 0.46 & 0.63 & 0.10 & 0 & 0.10 & 0.20 & 0.49 & 0.50 & 0.43 & 0.47 & 0.50 & 0.43 & 0.58 \\
        & \bfseries{PCS}  & 0.37 & 0.00 & 0.17 & 0.53 & 0.32 & 0.00 & 0.13 & 0.51 & 0.36 & 0.00 & 0.17 & 0.53 & 0.00 & 0 & 0.00 & 0.00 & 0.27 & 0.00 & 0.10 & 0.20 & 0.00 & 0.06 & 0.46 \\
        \midrule
        \multirow{3}{*}{\rotatebox{90}{\boldsymbol{$T\odot H$}}}
        & \bfseries{mean} & 4.48 & 2.72 & 3.16 & 3.62 & 4.42 & 2.69 & 3.13 & 3.57 & 4.48 & 2.72 & 3.15 & 3.61 & 2.53 & 2.00 & 2.11 & 2.20 & 4.27 & 2.69 & 3.10 & 4.27 & 2.64 & 3.04 & 3.46 \\
        & \bfseries{sd}   & 3.12 & 0.59 & 0.63 & 0.74 & 3.05 & 0.58 & 0.63 & 0.70 & 3.12 & 0.59 & 0.64 & 0.72 & 1.22 & 0.28 & 0.37 & 0.47 & 2.78 & 0.58 & 0.63 & 2.98 & 0.58 & 0.62 & 0.64 \\
        & \bfseries{PCS}  & 0.15 & 0.07 & 0.29 & 0.53 & 0.14 & 0.06 & 0.27 & 0.54 & 0.15 & 0.07 & 0.29 & 0.54 & 0.06 & 0.00 & 0.00 & 0.01 & 0.13 & 0.06 & 0.25 & 0.13 & 0.05 & 0.21 & 0.51 \\
        \bottomrule
    \end{tabular}
\end{table}

Our second experiment (see Table~\ref{table:2}) is in a relatively easier setting: we have $k = 4$ with $\lambda_4 = 1.5$. The scree plot estimator catches up with the AIC-type estimators. As $\lambda_4 > \lambda_{\gamma}$, $\gaic{\gamma}$ also performs quite well. Notably, both $\aic$ and $\gaic{2 + \delta_N}$ perform much worse in S-4 than S-1 or S-3, in contrast with the results in the first experiment. We again observe degraded performance for all the estimators in S-1 under the $T\odot H$ noise profile.

\begin{table}[!t]
    \centering
    \setlength{\tabcolsep}{0.22em}
    \caption{$k = 4$, $(\lambda_1, \lambda_2, \lambda_3, \lambda_4)^\top = (10, 3, 1.5, 1.5)^\top$, $\sigma^2 = 1$. $N = 1000$, $\delta_N = \frac{0.1}{\sqrt{N}}$, $\lambda_{2 + \delta_N} = 1.04$, $\gamma = 2.15$, $\lambda_{\gamma} = 1.31$. S-1: oracle $\sigma^2$; S-2: $\hat{\sigma^2_0}$; S-3: $\hat{\sigma^2_{\alpha}}$ with $\alpha = 0.1$; S-4: unknown $\sigma^2$.}
    \label{table:2}
    \fontsize{6}{8}\selectfont
    \begin{tabular}{l|l*{4}{c}|*{4}{c}|*{4}{c}|*{4}{c}|*{3}{c}|*{4}{c}}
        \toprule
        \multicolumn{2}{c}{} & \multicolumn{4}{c}{$\boldsymbol{\aic}$} & \multicolumn{4}{c}{$\boldsymbol{\gaic{2 + \delta_N}}$} & \multicolumn{4}{c}{$\boldsymbol{\saic}$} & \multicolumn{4}{c}{$\boldsymbol{\gaic{\gamma}}$} & \multicolumn{3}{c}{$\boldsymbol{\scree}$} & \multicolumn{4}{c}{$\boldsymbol{\saic^{(\mathrm{ad})}}$} \\
        \midrule
        \multicolumn{2}{c}{} & S-1 & S-2 & S-3 & S-4 & S-1 & S-2 & S-3 & S-4 & S-1 & S-2 & S-3 & S-4 & S-1 & S-2 & S-3 & S-4 & S-1 & S-2 & S-3 & S-1 & S-2 & S-3 & S-4 \\
        \midrule
        \multirow{3}{*}{\rotatebox{90}{\bfseries{GOE}}}
        & \bfseries{mean} & 4.06 & 3.79 & 4.03 & 4.32 & 4.06 & 3.79 & 4.03 & 4.27 & 4.06 & 3.79 & 4.02 & 4.22 & 3.98 & 2.68 & 3.98 & 4 & 4.04 & 3.79 & 4.02 & 4.02 & 3.74 & 4.01 & 4.08 \\
        & \bfseries{sd}   & 0.24 & 0.41 & 0.17 & 0.51 & 0.24 & 0.41 & 0.17 & 0.49 & 0.24 & 0.41 & 0.14 & 0.42 & 0.14 & 0.55 & 0.14 & 0 & 0.20 & 0.41 & 0.14 & 0.14 & 0.44 & 0.10 & 0.27 \\
        & \bfseries{PCS}  & 0.94 & 0.79 & 0.97 & 0.70 & 0.94 & 0.79 & 0.97 & 0.75 & 0.94 & 0.79 & 0.98 & 0.78 & 0.98 & 0.04 & 0.98 & 1 & 0.96 & 0.79 & 0.98 & 0.98 & 0.74 & 0.99 & 0.92 \\
        \midrule
        \multirow{3}{*}{\rotatebox{90}{\bfseries{Rad.}}}
        & \bfseries{mean} & 4.03 & 3.94 & 4 & 4.20 & 4.03 & 3.94 & 4 & 4.16 & 4.02 & 3.94 & 4 & 4.14 & 4 & 2.53 & 4 & 4 & 4.01 & 3.94 & 4 & 4 & 3.88 & 4 & 4.04 \\
        & \bfseries{sd}   & 0.17 & 0.24 & 0 & 0.43 & 0.17 & 0.24 & 0 & 0.37 & 0.14 & 0.24 & 0 & 0.35 & 0 & 0.50 & 0 & 0 & 0.10 & 0.24 & 0 & 0 & 0.33 & 0 & 0.20 \\
        & \bfseries{PCS}  & 0.97 & 0.94 & 1 & 0.81 & 0.97 & 0.94 & 1 & 0.84 & 0.98 & 0.94 & 1 & 0.86 & 1 & 0.00 & 1 & 1 & 0.99 & 0.94 & 1 & 1 & 0.88 & 1 & 0.96 \\
        \midrule
        \multirow{3}{*}{\rotatebox{90}{\boldsymbol{$T\odot H$}}}
        & \bfseries{mean} & 5.53 & 3.75 & 4 & 4.37 & 5.47 & 3.75 & 4 & 4.29 & 5.44 & 3.73 & 4 & 4.24 & 4.18 & 2.71 & 3.92 & 3.96 & 5.28 & 3.73 & 4 & 5.20 & 3.58 & 4 & 4.07 \\
        & \bfseries{sd}   & 2.90 & 0.52 & 0 & 0.54 & 2.82 & 0.52 & 0 & 0.50 & 2.82 & 0.55 & 0 & 0.47 & 1.00 & 0.69 & 0.31 & 0.24 & 2.52 & 0.53 & 0 & 2.58 & 0.62 & 0 & 0.26 \\
        & \bfseries{PCS}  & 0.61 & 0.79 & 1 & 0.66 & 0.62 & 0.79 & 1 & 0.73 & 0.64 & 0.78 & 1 & 0.78 & 0.84 & 0.13 & 0.93 & 0.97 & 0.64 & 0.77 & 1 & 0.68 & 0.65 & 1 & 0.93 \\
        \bottomrule
    \end{tabular}
\end{table}

\subsection{Comparing \texorpdfstring{$\boldsymbol{\aic}$}{} and \texorpdfstring{$\boldsymbol{\gaic{2 + \delta_N}}$}{}}
In Table~\ref{table:3}, we compare the performance of $\aic$ and $\gaic{2 + \delta_N}$ as $N$ varies. We use the same set-up as in our second experiment (except that we use a slightly higher value of $\delta_N = \frac{0.5}{\sqrt{N}}$) and report the results for oracle $\sigma$, i.e. S-1, in Table~\ref{table:3}-(A), and for unknown $\sigma$, i.e. S-4, in Table~\ref{table:3}-(B). The results favour $\gaic{2 + \delta_N}$ in accordance with our theoretical results, the difference being much more pronounced in the setting of unknown $\sigma$.

\begin{table}[!htb]
    \caption{$k = 4$, $(\lambda_1, \lambda_2, \lambda_3, \lambda_4)^\top = (10, 3, 1.5, 1.5)^\top$, $\sigma^2 = 1$. $\delta_N = \frac{0.5}{\sqrt{N}}$.}
    \label{table:3}
    \begin{subtable}{.5\linewidth}
      \fontsize{8}{10}\selectfont
      \centering
        \caption{$\sigma$ known}
        \begin{tabular}{l|l*{5}{c}}
            \toprule
            \multicolumn{1}{c}{} & $N$ & 1000 & 2000 & 3000 & 4000 & 5000 \\
            \midrule
            \multirow{3}{*}{$\boldsymbol{\aic}$}
            & \bfseries{mean} & 4.06 & 4.12 & 4.05 & 4.07 & 4.15 \\
            & \bfseries{sd}   & 0.24 & 0.33 & 0.22 & 0.26 & 0.39 \\
            & \bfseries{PCS}  & 0.94 & 0.88 & 0.95 & 0.93 & 0.86 \\
            \midrule
            \multirow{3}{*}{$\boldsymbol{\gaic{2 + \delta_N}}$}
            & \bfseries{mean} & 4.03 & 4.03 & 4.01 & 4.04 & 4.02 \\
            & \bfseries{sd}   & 0.17 & 0.17 & 0.10 & 0.20 & 0.14 \\
            & \bfseries{PCS}  & 0.97 & 0.97 & 0.99 & 0.96 & 0.98 \\
            \bottomrule
        \end{tabular}
    \end{subtable}%
    \begin{subtable}{.5\linewidth}
      \fontsize{8}{10}\selectfont
      \centering
        \caption{$\sigma$ unknown}
        \begin{tabular}{l|l*{5}{c}}
            \toprule
            \multicolumn{1}{c}{} & $N$ & 1000 & 2000 & 3000 & 4000 & 5000 \\
            \midrule
            \multirow{3}{*}{$\boldsymbol{\aic}$}
            & \bfseries{mean} & 4.32 & 4.24 & 4.30 & 4.26 & 4.29 \\
            & \bfseries{sd}   & 0.51 & 0.43 & 0.46 & 0.48 & 0.48 \\
            & \bfseries{PCS}  & 0.70 & 0.76 & 0.70 & 0.76 & 0.72 \\
            \midrule
            \multirow{3}{*}{$\boldsymbol{\gaic{2 + \delta_N}}$}
            & \bfseries{mean} & 4.09 & 4.11 & 4.03 & 4.06 & 4.10 \\
            & \bfseries{sd}   & 0.29 & 0.31 & 0.17 & 0.24 & 0.33 \\
            & \bfseries{PCS}  & 0.91 & 0.89 & 0.97 & 0.94 & 0.91 \\
            \bottomrule
        \end{tabular}
    \end{subtable} 
\end{table}

\subsection{Estimating the number of communities}
\begin{figure}
    \centering
\begin{tabular}{cc}
    \includegraphics[width=.48\textwidth]{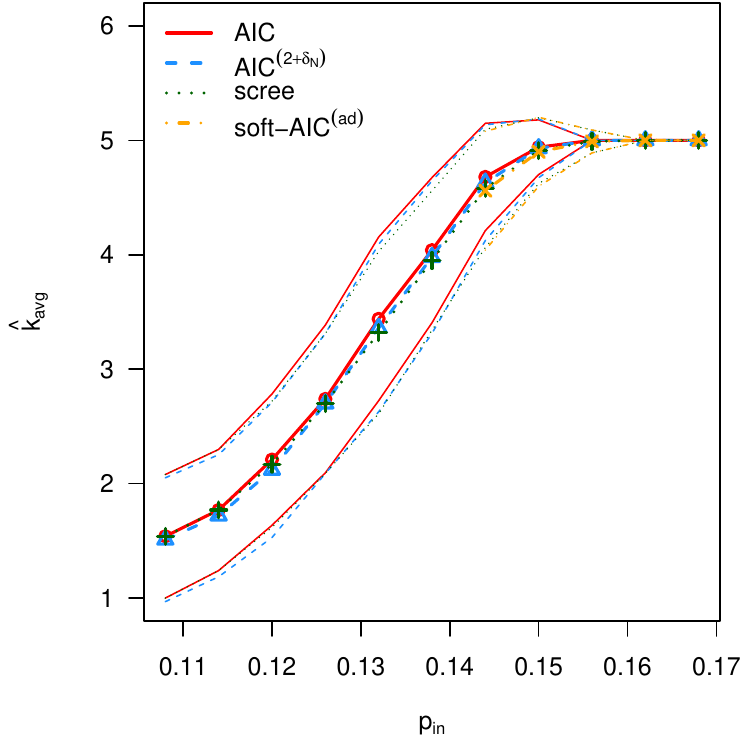} & \includegraphics[width=.48\textwidth]{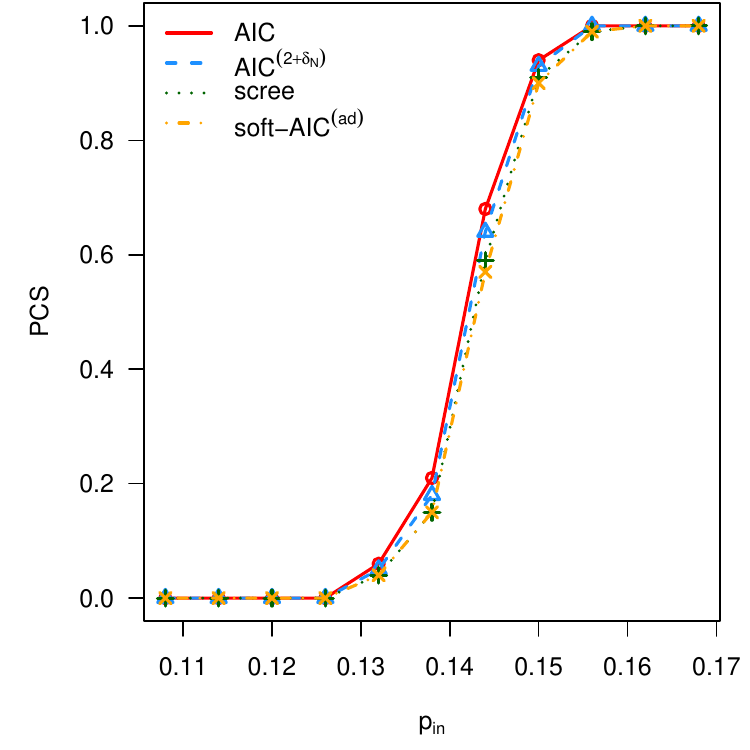} 
\end{tabular}
\caption{Comparison of various estimators of $k$ under SBM with equal community sizes. Here $N = 500$, $k = 5$, $p_{\across} = 0.06$. The results are based on $100$ Monte Carlo runs.}
\label{fig:SBM}
\end{figure}

In this section, we present some empirical results on estimating the number of communities in an SBM using the criterion \eqref{eq:aic_comm}. We consider SBMs on $n = 500$ vertices with $k = 5$ equal-sized communities with $p_{\within}$ varying between $0.11$ and $0.17$ and $p_{\across} = 0.06$. Figure~\ref{fig:SBM} shows the performances of the various estimators. They seem to have comparable performance with $\aic$ and $\gaic{2 + \delta_N}$ showing a slight edge over the scree plot estimator and the $\saic^{(\mathrm{ad})}$.

\begin{table}[h]
    \centering
    \setlength{\tabcolsep}{0.3em}
    \caption{Results (based on $100$ Monte Carlo runs) for SBM with $N = 500$, $k = 5$, and block connection probability matrix $\Pi$ as in \eqref{eq:block-conn-mat}; $\delta_N = \frac{0.1}{\sqrt{N}}$.}
    \label{table:4}
    \fontsize{8}{10}\selectfont
    \begin{tabular}{l|c|c|c|c}
    \toprule
      & $\boldsymbol{\aic}$ & $\boldsymbol{\gaic{2 + \delta_N}}$ & $\boldsymbol{\scree}$ & $\boldsymbol{\saic^{(\mathrm{ad})}}$ \\
    \midrule
        \textbf{mean} & 5.08 & 4.99 & 4.92 & 4.95 \\
        \textbf{sd}   & 0.69 & 0.66 & 0.66 & 0.65 \\
        \textbf{PCS}  & 0.55 & 0.60 & 0.59 & 0.60 \\
    \bottomrule
    \end{tabular}
\end{table}

We also consider a setting where the within-community and across-community connection probabilities are different. We take $k = 5$ equal sized communities and a $k\times k$ block connection probability matrix
\begin{equation}\label{eq:block-conn-mat}
    \Pi = \begin{pmatrix}
        0.073 & 0.015 & 0.005 & 0.013 & 0.017 \\
        0.015 & 0.079 & 0.004 & 0.006 & 0.004 \\
        0.005 & 0.004 & 0.078 & 0.009 & 0.008 \\
        0.013 & 0.006 & 0.009 & 0.063 & 0.013 \\
        0.017 & 0.004 & 0.008 & 0.013 & 0.072
    \end{pmatrix}.
\end{equation}
The results are reported in Table~\ref{table:4}.

Finally, we apply the methods to five real datasets (available at \url{https://websites.umich.edu/~mejn/netdata/}). Despite the apparent crudeness of the methods inherent in the construction of $\hat{M}$ (e.g., not accounting for possible inhomogeneities in the block connection probabilities or for degree heterogeneities), they seem to perform reasonably well in these examples (see Table~\ref{table:5}).
\begin{table}
    \centering
    \caption{Estimating the number of communities in several real networks. $\delta_N = \frac{0.1}{\sqrt{N}}$.}
    \label{table:5}
    \fontsize{8}{10}\selectfont
    \begin{tabular}{r|cccccc}
        \toprule
        Data & $N$ & Ground truth $k$ & $\aic$ & $\gaic{2 + \delta_N}$ & $\scree$ & $\saic^{(\mathrm{ad})}$ \\
        \midrule
        dolphins       & 62  & 2, 4 \cite{Lusseau2003, liu2016discovering} & 4  & 4  & 4  & 4  \\
        les miserables & 77  & 6 \cite{newman2016estimating}               & 5  & 5  & 5  & 5  \\
        football       & 115 & 11 \cite{newman2016estimating}              & 10 & 10 & 10 & 10 \\
        polbooks       & 105 & 3, 4 \cite{Newman2006}                      & 5  & 5  & 4  & 4  \\
        karate         & 34  & 2 \cite{zachary1977information, Newman2006} & 3  & 3  & 3  & 3  \\
        \bottomrule
    \end{tabular}
\end{table}

\section{Discussion}\label{sec:conc}
The AIC-type model selection criteria developed in this paper are quite flexible in terms of their applicability. As random matrix theoretic results become available under more general spiked models (e.g., correlated noise models, general SBMs, or degree corrected SBMs, etc.), asymptotic consistency properties of the proposed criteria may be readily derived. Nevertheless, there remain a number of relevant questions not answered in this article that may constitute future research directions. The more prominent ones include extending our results to allow $k$ to grow with $N$, establishing strong consistency of $\gaic{2 + \delta_N}$, investigation of a possible data-driven choice of $\delta_N$ and establishing non-asymptotic error bounds.

\section*{Acknowledgements}
We thank the anonymous referees and the associate editor for their many constructive comments and suggestions, which have led to a significantly improved article. This research was partially supported by the INSPIRE research grant DST/INSPIRE/04/2018/002193 from the Department of Science and Technology, Government of India, and a Start-Up Grant from Indian Statistical Institute.

\bibliographystyle{alpha-abbrv}
\bibliography{swm-aic.bib}

\newpage
\appendix
\section{Proofs}\label{sec:proofs}
\subsection{Proofs for spiked Wigner}
Note that as long as $j = o(N)$,
\begin{equation}\label{eq:sigma2hat-conv-sigma2}
    \hat{\sigma^2_j} \convas \int x^2 \varrho_{\mathrm{sc}}(x; \sigma^2) \, dx = \sigma^2.
\end{equation}
Because of this, and the fact that we have only one extra parameter in the case of unknown $\sigma$, the proofs of Theorems~\ref{thm:gamma-aic}, \ref{thm:almost-aic-weak-consistency} and \ref{thm:soft-aic} will essentially be the same regardless of whether $\sigma$ is known or unknown. For simplicity, we will only write down the details for the case of known $\sigma$.

\begin{proof}[Proof of Theorem~\ref{thm:gamma-aic}]
    For $j < k$, we have, using Assumption~\ref{assmp:bbp}-(a), that
\begin{align*}
    \frac{1}{N}(\gaic{\gamma}_j - \gaic{\gamma}_k) &=  \frac{1}{2\sigma^2} \sum_{i = j + 1}^k \ell_i^2 - \gamma (k - j) \bigg(1 - \frac{k + j - 1}{2N} \bigg) \\
              &\convas \frac{1}{2\sigma^2} \sum_{i = j + 1}^k (\psi_{\sigma}(\lambda_i))^2 - \gamma (k - j) \\
              &\ge \frac{1}{2\sigma^2} (k - j) (\psi_{\sigma}(\lambda_k))^2 -  \gamma (k - j) \\
              &= \frac{(k - j)}{2\sigma^2} [(\psi_{\sigma}(\lambda_k))^2 -  2\gamma \sigma^2].
\end{align*}
Note that for $\gamma \le 2$,
\begin{align*}
    (\psi_{\sigma}(\lambda_k))^2 - 2\gamma\sigma^2 &> 4\sigma^2 - 2\gamma \sigma^2 \\
                                                   &= 2\sigma^2 (2 - \gamma) \ge 0,
\end{align*}
where we have used the fact that $\lambda_k > \sigma$ implies that $\psi_\sigma(\lambda_k) > 2 \sigma$. It follows that
\[
    \liminf_{N \to \infty}\hat{k}_{\gamma} \ge k \text{ a.s.}
\]
for any $\gamma \le 2$. This establishes (a).

On the other hand, for $j > k$, Assumption~\ref{assmp:bbp}-(b) gives that
\begin{align*}
    \frac{1}{N}(\gaic{\gamma}_j - \gaic{\gamma}_k) &= -\frac{1}{2\sigma^2} \sum_{i = k + 1}^j \ell_i^2 + \gamma (j - k) \bigg(1 - \frac{k + j - 1}{2N} \bigg) \\
              &\convas -\frac{1}{2\sigma^2} \sum_{i = k + 1}^j 4 \sigma^2 + \gamma (j - k) \\
              &= (\gamma - 2) (j - k) > 0
\end{align*}
if $\gamma > 2$. It follows that
\[
    \limsup_{N \to \infty}\hat{k}_{\gamma} \le k \text{ a.s.}
\]
for any $\gamma > 2$. This establishes (b).

Finally, for $\gamma > 2$, $\hat{k}_{\gamma}$ will be strongly consistent provided
\[
    \psi_{\sigma}(\lambda_k) > \sqrt{2\gamma} \sigma > 2\sigma
\]
which is equivalent to
\[
    \lambda_k > \psi_{\sigma}^{-1}(\sqrt{2\gamma}\sigma) > \psi_{\sigma}^{-1}(2\sigma) = \sigma.
\]
This establishes (c).
\end{proof}

We now sketch how one can relax the condition that $q = O(1)$ under Assumption~\ref{assmp:bbp+rigidity} (with $q_N = N$). For this, we shall need the following result.
\begin{lemma}\label{lem:I_gamma}
    For $\alpha \in [0, 1]$, let $q_{1 - \alpha}$ denote the $(1 - \alpha)$-th quantile of the standard semi-circle law, i.e. $\int_{q_{1 - \alpha}}^2 \varrho_{\semicirc}(t; 1) \, dt = \alpha$. Define
\[
    I_{\gamma}(\alpha) = \gamma \alpha (1 - \alpha / 2) -\frac{1}{2} \int_{q_{1 - \alpha}}^2 t^2 \varrho_{\semicirc}(t; 1) \, dt.
\]
Then for $\gamma > 2$, we have that $I_{\gamma}(\alpha) > 0$ for all $\alpha \in (0, 1)$.
\end{lemma}
\begin{proof}
    Note that $I_{\gamma}(0) = 0$ and $I_{\gamma}(1) = \frac{\gamma - 1}{2} > 0$. Also, for $\alpha \in (0, 1)$,
    \[
        I_{\gamma}'(\alpha) = \gamma - \gamma \alpha + \frac{1}{2} q_{1 - \alpha}^2 \varrho_{\semicirc}(q_{1 - \alpha}; 1) \frac{d q_{1 - \alpha}}{d\alpha} = \gamma - \gamma \alpha - \frac{1}{2} q^2_{1 - \alpha}.
    \]
    Further, $I_{\gamma}'(0+) = \gamma - 2 > 0$, $I_{\gamma}'(1-) = -2 < 0$. Note also that $I_{\gamma}'(1/2) = \gamma / 2 > 0$. Finally, note that
    \[
        I_{\gamma}''(\alpha) = -\gamma + \frac{2 \pi q_{1 - \alpha}}{\sqrt{4 - q_{1 - \alpha}^2}}.
    \]
    It follows that $I_{\gamma}''(\alpha) = 0$ if and only if $q_{1 - \alpha} = \frac{2 \gamma}{\sqrt{\gamma^2 + 4\pi^2}}$, i.e.
    \[
        \alpha = \alpha_*(\gamma) := \int_{\frac{2\gamma}{\sqrt{\gamma^2 + 4 \pi^2}}}^2 \varrho_{\semicirc}(t; 1) \, dt.
    \]
    Also, $I_{\gamma}''(\alpha) > 0$ on $(0, \alpha_*(\gamma))$ and $I_{\gamma}''(\alpha) < 0$ on $(\alpha_*(\gamma), 1)$. Therefore $I_{\gamma}'$ first (strictly) increases up to $\alpha_*(\gamma)$ and then (strictly) decreases. Therefore, there exists a unique location $\tilde{\alpha}(\gamma)$  with $\tilde{\alpha}(\gamma) > 1/2 > \alpha_*(\gamma)$ at which $I_{\gamma}'$ vanishes. At this point $I_{\gamma}$ is at its global maximum. After this point $I_{\gamma}$ strictly decreases. Since $I_{\gamma}(1) > 0$, we conclude that $I_{\gamma}(\alpha) > 0$ for all $\alpha \in (0, 1)$. 
\end{proof}

Now let us revisit the proof of Theorem~\ref{thm:gamma-aic}-(b). We must now handle the cases (i) $j / N \to 0$ with $j - k \to \infty$ and (ii) $j / N \to \alpha > 0$.

If $j / N \to 0$ with $j - k \to \infty$, then by Assumption~\ref{assmp:bbp+rigidity}-(b) with $q_N = N$,
\begin{align*}
    \frac{1}{N (j - k)}(\gaic{\gamma}_j - \gaic{\gamma}_k) &= - \frac{1}{2\sigma^2} \frac{1}{j - k} \sum_{i = k + 1}^j \ell_i^2 + \gamma \bigg(1 - \frac{k + j - 1}{2N} \bigg) \\
                                                           &\convas \gamma - 2 > 0.
\end{align*}
If, on the other hand, $ j > k$ and $j / N \to \alpha$, then note that by Assumption~\ref{assmp:bbp+rigidity}-(b) with $q_N = N$,
\begin{align*}
    \frac{1}{N^2}(\gaic{\gamma}_j - \gaic{\gamma}_k) &= -\frac{1}{2\sigma^2} \frac{1}{N} \sum_{i = k + 1}^j \ell_i^2 + \gamma \cdot \frac{j - k}{N} \bigg(1 - \frac{k + j - 1}{2N} \bigg) \\
                                                     &\convas -\frac{1}{2} \int_{q_{1 - \alpha}}^2 t^2 \varrho_{\semicirc}(t; 1) \, dt + \gamma \alpha (1 - \alpha / 2) \\
                                                     &= I_{\gamma}(\alpha) > 0,
\end{align*}
where the last inequality holds since $\gamma > 2$ (see Lemma~\ref{lem:I_gamma}). It follows that
\[
    \limsup_{N \to \infty}\hat{k}_{\gamma} \le k \text{ a.s.}
\]
for any $\gamma > 2$.

\begin{proof}[Proof of Theorem~\ref{thm:almost-aic-weak-consistency}]
    By Assumption~\ref{assmp:fluc-order}, $\tilde{\ell}_{j} = N^{2/3}(\ell_{j} - 2\sigma) = O_P(1)$ for $k < j < q$. Therefore, for $k < j < q$,
\begin{align*}
    \frac{1}{N}(\gaic{\gamma}_j - \gaic{\gamma}_k) &= -\frac{1}{2\sigma^2} \sum_{i = k + 1}^j \ell_i^2 + \gamma_N (j - k) \bigg(1 - \frac{k + j - 1}{2N} \bigg) \\
                                                   &= -\frac{1}{2\sigma^2} \sum_{i = k + 1}^j (2\sigma + N^{-2/3}\tilde{\ell_i})^2 + (2 + \delta_N) (j - k) \bigg(1 - \frac{k + j - 1}{2N} \bigg) \\
                                                   &= \delta_N (j - k) - \frac{1}{\sigma} \sum_{i = k + 1}^j N^{-2/3} \tilde{\ell}_i + O_P\bigg(\frac{1}{N}\bigg) \\
                                                   &\ge \delta_N (j - k) - \frac{1}{\sigma} (j - k) N^{-2/3} \tilde{\ell}_{k + 1} + O_P\bigg(\frac{1}{N}\bigg) \\
                                                   &= \delta_N(j - k) (1 + o_P(1)),
\end{align*}
provided $\delta_N \gg N^{-2/3}$. Thus $\bP(\gaic{\gamma}_j - \gaic{\gamma}_k \ge \epsilon N\delta_N(j - k)) \rightarrow 1$ for any $\epsilon \in (0, 1)$. It follows that $\bP(\hat{k}_\gamma \le k) \rightarrow 1$. Combining this with Theorem~\ref{thm:gamma-aic}-(a), we get the desired weak-consistency result.
\end{proof}

\begin{proof}[Proof of Theorem~\ref{thm:soft-aic}]
    Let $\xi_0 := \liminf_{N \to \infty} \hat{\xi} > 0$. There is a set $E$ with $\bP(E) = 0$, such that for all $\omega \in E^c$, we have the following:
\begin{enumerate}
    \item[(i)]  For $j < k$,
        \[
            \frac{1}{N}(\aic_j - \aic_k) \to \sum_{i = j + 1}^k \frac{1}{2\sigma^2}[\psi_{\sigma}(\lambda_i)^2 - 4\sigma^2] \ge (k - j) \xi_k \ge \xi_k.
        \]
        \item[(ii)] For $k < j < q$,
        \[
            \frac{1}{N}(\aic_j - \aic_k) \to 0.
        \]
    \item[(iii)] For all sufficiently large $N$ (depending possibly on $\omega$), $\frac{\xi_0}{2} < \hat{\xi} < \frac{6}{5}\xi_k$.
\end{enumerate}
Thus, for any $\omega \in E^c$, we can choose $N_0(\omega)$ large enough such that for all $N \ge N_0(\omega)$, the following hold:
\begin{enumerate}
    \item[(a)] For all $j < k$, $\frac{1}{N}(\aic_j - \aic_k) \ge \frac{4\xi_k}{5} > \frac{2\hat{\xi}}{3}$.

    \item[(b)] For all $k < j < q$, $\frac{1}{N}(\aic_j - \aic_k) < \frac{\xi_0}{6} < \frac{\hat{\xi}}{3}$.
\end{enumerate}
Together, (a) and (b) imply that
\[
    \frac{1}{N}|\aic_k - \min_{0 \le j' < q} \aic_{j'}| < \frac{\hat{\xi}}{3},
\]
and for $j < k$,
\[
    \frac{1}{N}|\aic_j - \min_{0 \le j' < q} \aic_{j'}| > \frac{1}{N}|\aic_j - \aic_k| - \frac{1}{N}|\aic_k - \min_{0 \le j' < q} \aic_{j'}| > \frac{2\hat{\xi}}{3} - \frac{\hat{\xi}}{3} = \frac{\hat{\xi}}{3}.
\]
 Therefore $\hat{k}_{2, \,\soft} = k$ for all $N \ge N_0(\omega)$. This completes the proof.
\end{proof}

\begin{proof}[Proof of Proposition~\ref{prop:scree}]
    This is an immediate consequence of the following facts:
    \begin{enumerate}
        \item[(a)] For $k < j < q$, $\limsup_{N \to \infty}\frac{\ell_j}{2\hat{\sigma}} \le 1$ almost surely.
        \item[(b)] For $j \le k$, we have $\ell_j \convas \psi_{\sigma}(\lambda_j) > 2 \sigma$ so that $\lim_{N \to \infty} \frac{\ell_j}{2\hat{\sigma}} > 1$.
    \end{enumerate}
\end{proof}

\begin{proof}[Proof of Theorem~\ref{thm:two-step}]
We know that $\hat{k}_{\scree} \convas k$. Thus there is an event $N$ with $\bP(N) = 0$ such that for any $\omega \notin N$, for all large enough $n$, one has $\hat{k}_{\scree}(\omega) = k$ and consequently, $\hat{\xi}_{\hat{k}_{\scree}} = \hat{\xi}_k$ . We also have that $\hat{\xi}_k \convas \xi_k$. It follows that
\[
    \hat{\xi}_{\hat{k}_{\scree}} \convas \xi_k.
\]
Since $\hat{\xi}_1 \convas \xi_1$, we conclude that
\[
    \hat{B} \convas \frac{\xi_1}{\xi_k}.
\]
It follows that
\[
    \hat{\xi}_{\scree} = \frac{1}{q\hat{B}} \sum_{j = 1}^q \hat{\xi}_j \convas \frac{1}{\frac{q\xi_1}{\xi_k}} \sum_{j = 1}^k \xi_j \in (0, \xi_k).
\]
Thus the threshold $\hat{\xi}_{\scree}$ satisfies condition \eqref{eq:soft-aic-thres-prop}. Hence Theorem~\ref{thm:soft-aic} applies.
\end{proof}

\subsection{Proofs for the SBM}
We begin with the fact that by the Chernoff bound, for any $0 < \delta < 1$, one has 
\[
    \bP(|\hat{p}_{\avg} - p_{\avg}| \ge \delta p_{\avg}) \le 2 \exp(-\delta^2 N^2 p_{\avg} / 6).
\]
Therefore, when $p_{\avg} \gg N^{-2\beta}$, $\beta \in (0, 1)$, we have
\[
    \bP(|\hat{p}_{\avg} - p_{\avg}| \ge \delta p_{\avg}) \ll 2 \exp(-\delta^2 N^{2(1 - \beta)} / 6).
\]
Choosing $\delta = \frac{\sqrt{\log N}}{N^{1 - \beta}}$, we have that with probability at least $1 - O(N^{-2})$,
\[
    |\hat{p}_{\avg} - p_{\avg}| < p_{\avg} \frac{\sqrt{\log N}}{N^{1 - \beta}}. 
\]
\begin{lemma}
    We have
    \[
        \big|\frac{\hat{\varsigma}}{\varsigma} - 1\big| \le \frac{N|\hat{p}_{\avg} - p_{\avg}|}{\varsigma^2} + \frac{k \theta_N^2}{N}.
    \]
\end{lemma}
\begin{proof}
We first note that
\[
    |\frac{\hat{\varsigma}}{\varsigma} - 1| = \frac{|(\frac{\hat{\varsigma}}{\varsigma})^2 - 1|}{\frac{\hat{\varsigma}}{\varsigma} + 1} \le \big|\bigg(\frac{\hat{\varsigma}}{\varsigma}\bigg)^2 - 1\big|.
\]
Therefore it is enough to prove that
\[
    \bigg|\bigg(\frac{\hat{\varsigma}}{\varsigma}\bigg)^2 - 1\bigg| \le \frac{N|\hat{p}_{\avg} - p_{\avg}|}{\varsigma^2} + \frac{k \theta_N^2}{N}.
\]
Towards that end, note that
\begin{align*}
    |\hat{\varsigma}^2 - \varsigma^2| &= N \bigg|\hat{p}_{\avg}(1 - \hat{p_{\avg}}) - \frac{p_{\within}(1 - p_{\within}) + (k - 1) p_{\across}(1 - p_{\across})}{k}\bigg| \\
                                &\le N \bigg|\hat{p}_{\avg}(1 - \hat{p_{\avg}}) - p_{\avg} (1 - p_{\avg})\bigg| \\
                                &\hspace{4em}+ N\bigg|p_{\avg} (1 - p_{\avg}) - \frac{p_{\within}(1 - p_{\within}) + (k - 1) p_{\across}(1 - p_{\across})}{k}\bigg| \\
                                &\le N|\hat{p}_{\avg} - p_{\avg}| + \frac{\varsigma^2 k\theta_N^2}{N},
\end{align*}
where we have used the following bound
\begin{align*}
    N\bigg|p_{\avg} (1 - p_{\avg}) &- \frac{p_{\within}(1 - p_{\within}) + (k - 1) p_{\across}(1 - p_{\across})}{k}\bigg| \\
                                   &=N\bigg|\frac{p_{\within}^2 + (k - 1)p_{\across}^2}{k} - p_{\avg}^2\bigg| \\
                                   &=N\bigg|\frac{p_{\within}^2 + (k - 1)p_{\across}^2}{k} - \frac{p_{\within}^2 + (k - 1)^2 p_{\across}^2 + 2 (k - 1) p_{\within} p_{\across}}{k^2}\bigg|\\
                                   &=\frac{N(k - 1)}{k^2}(p_{\within} - p_{\across})^2 \\
                                   &\le\frac{\varsigma^2 k\theta_N^2}{N}.
\end{align*}
This completes the proof.
\end{proof}

\begin{lemma}\label{lem:M-vs-Mhat}
    We have
    \[
    \|\hat{M} - M\|_{\op} \le \frac{|\frac{\hat{\varsigma}}{\varsigma} - 1| \cdot \|M\|_{\op} +  \frac{N|\hat{p}_{\avg} - p_{\avg}|}{\varsigma}}{1 - |\frac{\hat{\varsigma}}{\varsigma} - 1|}.
    \]
\end{lemma}
\begin{proof}
Note that
\[
    \hat{M} - M = \bigg(\frac{\varsigma}{\hat{\varsigma}} - 1\bigg) M + \frac{p_{\avg} - \hat{p}_{\avg}}{\hat{\varsigma}} J_N
\]
and consequently,
\begin{align*}
    \|\hat{M} - M\|_{\op} &\le |\frac{\varsigma}{\hat{\varsigma}} - 1| \cdot \|M\|_{\op} + \frac{N|\hat{p}_{\avg} - p_{\avg}|}{\hat{\varsigma}}.
\end{align*}
Now
\[
    \big|\frac{\varsigma}{\hat{\varsigma}} - 1\big| = \frac{|\frac{\hat{\varsigma}}{\varsigma} - 1|}{\frac{\hat{\varsigma}}{\varsigma}} = \frac{|\frac{\hat{\varsigma}}{\varsigma} - 1|}{1 + \frac{\hat{\varsigma}}{\varsigma} - 1} \le \frac{|\frac{\hat{\varsigma}}{\varsigma} - 1|}{1 - |\frac{\hat{\varsigma}}{\varsigma} - 1|}
\]
and
\[
    \frac{N|\hat{p}_{\avg} - p_{\avg}|}{\hat{\varsigma}} \le \frac{N|\hat{p}_{\avg} - p_{\avg}|}{\varsigma(1 - |\frac{\hat{\varsigma}}{\varsigma} - 1|)}.
\]
The proof follows.
\end{proof}

\begin{proof}[Proof of Proposition~\ref{prop:M-vs-Mhat}]
First note that
\[
    |\frac{\hat{\varsigma}}{\varsigma} - 1| \le \frac{N|\hat{p}_{\avg} - p_{\avg}|}{\varsigma^2} + \frac{k \theta_N^2}{N} \le  C_{\delta_0} \frac{|\hat{p}_{\avg} - p_{\avg}|}{p_{\avg}} + \frac{k \theta_N^2}{N} \le C_{\delta_0} \frac{\sqrt{\log N}}{N^{1 - \beta}} + \frac{k \theta_N^2}{N}
\]
with probability at least $1 - O(N^{-2})$. Consequently, due to the assumption (i),
\[
    |\frac{\hat{\varsigma}}{\varsigma} - 1| \convas 0.
\]
We also note that
\[
    \|M\|_{\op} \le \theta_N + \|M - \bE M\|_{\op}.
\]
It follows that with probability at least $1 - O(N^{-2})$,
\begin{align*}
    |\frac{\hat{\varsigma}}{\varsigma} - 1| \cdot \frac{\|M\|_{\op}}{\theta_N} &\le |\frac{\hat{\varsigma}}{\varsigma} - 1| +  C_{\delta_0} \frac{\sqrt{\log N}}{N^{1 - \beta}} \frac{\|M - \bE M\|_{\op}}{\theta_N} + \frac{k \theta_N \|M - \bE M\|_{\op}}{N} \\
                                                                                                          &\le |\frac{\hat{\varsigma}}{\varsigma} - 1| +  C_{\delta_0} \frac{\sqrt{\log N}}{N^{1 / 2 - \beta}} \frac{1}{\theta_N} \frac{\|M - \bE M\|_{\op}}{\sqrt{N}} + \frac{k \theta_N}{\sqrt{N}} \frac{\|M - \bE M\|_{\op}}{\sqrt{N}}.
\end{align*}
Therefore under our assumptions,
\[
    |\frac{\hat{\varsigma}}{\varsigma} - 1| \cdot \frac{\|M\|_{\op}}{\theta_N} \convas 0.
\]
Also, with probability $\ge 1 - O(N^{-2})$, we have
\[
    \frac{N|\hat{p}_{\avg} - p_{\avg}|}{\varsigma} \le C_{\delta_0} \varsigma \frac{|\hat{p}_{\avg} - p_{\avg}|}{p_{\avg}} \le C_{\delta_0} \sqrt{N p_{\avg}} \frac{\sqrt{\log N}}{N^{1 - \beta}} \ll 1,
\]
since $\beta < 1/2$. Thus $\frac{N|\hat{p}_{\avg} - p_{\avg}|}{\varsigma} \convas 0$. From Lemma~\ref{lem:M-vs-Mhat}, it therefore follows that $\theta_N^{-1} \|\hat{M} - M\|_{\op} \convas 0$.  
\end{proof}

In view of Proposition~\ref{prop:BBP_Mhat}, the proof of Theorem~\ref{thm:gamma-aic-comm} is the same as that of Theorem~\ref{thm:gamma-aic} and hence is omitted.

\end{document}